\documentclass{amsart}
\usepackage{mathptmx, amssymb, enumerate, colonequals, amscd, bm, amsmath,mathtools}
\usepackage{color,xcolor}
\usepackage[colorlinks=true, pagebackref, hyperindex, citecolor=green, linkcolor=red]{hyperref}
\usepackage[capitalise]{cleveref}
\usepackage{subdepth, appendix}
\usepackage[margin=0.8in]{geometry}

\DeclareMathOperator{\mingens}{Mingens}
\DeclareMathOperator{\lcm}{lcm}
\DeclareMathOperator{\sign}{sgn}
\DeclareMathOperator{\scarfnumber}{Scarf-number}
\DeclareMathOperator{\rank}{rank}
\DeclareMathOperator{\image}{Im}
\DeclareMathOperator{\pd}{pd}

\newtheorem{theorem}{Theorem}[section]
\newtheorem{lemma}[theorem]{Lemma}
\newtheorem{proposition}[theorem]{Proposition}
\newtheorem{corollary}[theorem]{Corollary}

\theoremstyle{definition}
\newtheorem{definition}[theorem]{Definition}
\newtheorem{example}[theorem]{Example}

\theoremstyle{remark}
\newtheorem{remark}[theorem]{Remark}

\numberwithin{equation}{section}

\begin{document}

% \title[short text for running head]{full title}
\title{A family of simplicial resolutions which are DG-algebras}

\author{James Cameron}
\address{Department of Mathematics, University of Utah, 155 South 1400 East, Salt Lake City, UT~84112, USA}
\email{cameron@math.utah.edu}

\author{Trung Chau}
\address{Chennai Mathematical Institute, Chennai, Tamil Nadu 603103, India}
\email{chauchitrung1996@gmail.com}

\author{Sarasij Maitra}
\address{Department of Mathematics, University of Utah, 155 South 1400 East, Salt Lake City, UT~84112, USA}
\email{sarasij.maitra@utah.edu}

\author{Tim Tribone}
\address{Department of Mathematics, University of Utah, 155 South 1400 East, Salt Lake City, UT~84112, USA}
\email{tim.tribone@utah.edu}

\dedicatory{This paper is dedicated to  J\"urgen Herzog.}

\begin{abstract}
    Each monomial ideal over a polynomial ring admits a free resolution which has the structure of a DG-algebra, namely, the Taylor resolution. A pivot resolution of a monomial ideal, which we introduce, is a resolution that is always shorter than the Taylor resolution (unless the Taylor resolution is as short as possible) but still retains a DG-algebra structure. We study the basic properties of this family of resolutions including a characterization of when the construction is minimal. Following the work of Sobieska, we use the explicit nature of pivot resolutions to give formulae for the Eisenbud-Shamash construction of a free resolution of a given monomial ideal over complete intersections.
\end{abstract}
\maketitle

\section{Introduction}

Let $Q$ be a commutative ring and $I$ an ideal of $Q$. A recurring theme in the study of free resolutions of $Q/I$ over $Q$ is the contrast between highly structured resolutions, such as those with the structure of a differential graded (DG) algebra, and resolutions that are close to the minimal one. Over a polynomial ring, where resolutions can be taken to be finite, this dichotomy persists; the resolutions which are as short as possible may not support a multiplicative structure \cite{Avr81,Kat19} while resolutions which do have the structure of a DG-algebra are often far from being minimal \cite{Ge76}.

Classic examples of DG-algebra resolutions of $Q/I$ over $Q$ include the construction of Tate \cite{Tate57}, the minimal resolution of $Q/I$ if $I$ is generated by a regular sequence (in which case the minimal resolution is the Koszul complex), or if $\pd_Q(Q/I)\leq 3$ \cite{BE77}~(c.f.\cite[Proposition~2.1.4]{Avr98}). We refer to \cite[2.8]{Nasseh2021} and \cite{PeevaThesis, Sko11} for more known cases.

In the case that $Q$ is a polynomial ring over a field and $I$ is a monomial ideal, there always exists a free resolution of $Q/I$ over $Q$ which has the structure of a DG-algebra, namely, the Taylor resolution \cite{Tay66}, and indeed, it is usually highly non-minimal, especially as the number of generators of $I$ increases. On the other hand, there are now numerous examples in the literature of ideals whose minimal free resolutions do not support a DG-algebra structure. Of particular interest to us are the recent examples of Katth\"an \cite{Kat19} which show that Lyubeznik  resolutions \cite{Ly88} and the Scarf complex \cite{BPS98} need not support a multiplicative structure in general. These are improvements of Taylor resolutions, and are minimal under different~circumstances.

In this paper, we construct a new family of resolutions, called \textbf{pivot resolutions} (and more generally \textbf{pivot complexes}) using discrete Morse Theory; see \cite{BW02, CK24} or \cref{sec:background} for a brief introduction. They serve as a middle ground between these two extremes in the following sense: they are always shorter than the Taylor resolution (unless the Taylor resolution itself is minimal) and they always have the structure of a DG-algebra which is inherited from the Taylor resolution.

\begin{theorem}[Theorem ~\ref{thm:pivot-is-DG}]
    Let $Q$ be a polynomial ring over a field and $I$ a monomial ideal. Then $Q/I$ has a pivot resolution, and any such resolution has the structure of a DG-algebra over $Q$.
\end{theorem}

We also characterize which pivot complexes are resolutions in \cref{pivot}.
\begin{theorem}[Theorem ~\ref{thm:pivot-when-resolution}]
    Let $I$ be an ideal of $Q$ minimally generated by monomials $m_1,\dots, m_q$. Then the pivot complex corresponding to an increasing sequence $\{i_1< \cdots < i_l\}\subseteq \{1,2,\ldots,q\}$ for some $l<q$ is a resolution of $Q/I$ if and only if there exists an index $h\in \{1,2,\dots, q\}\setminus \{i_1,\dots, i_l\}$ such that $m_h \mid \lcm(m_{i_1},\dots, m_{i_l})$. 
\end{theorem}

Just like Lyubeznik resolutions and the Scarf complex, pivot resolutions are subcomplexes of the Taylor resolution in a canonical way. In particular, for each pivot resolution, there exists a Lyubeznik resolution such that
\[
    \text{Lyubeznik resolution } \subseteq \text{Pivot resolution } \subseteq \text{Taylor resolution}
    \]
as complexes (c.f. Remark \ref{rem:Lyubeznik-pivot-Taylor}). Katth\"an's example \cite[Theorem 5.1]{Kat19} is an instance where pivot resolutions are good middle ground for when one wants free resolutions that are close to the minimal one and still retain the DG-algebra structure. Furthermore, we show that unless the Taylor resolution is minimal, there always exists a pivot resolution shorter than the Taylor resolution itself. 

In \Cref{pivot}, we introduce the notion of Scarf-number which plays an important role in finding the ``smallest" pivot resolution. We also determine when a monomial ideal admits a minimal pivot resolution (see \cref{cor:Betti-minimal-pivot}).

Let $\mathfrak{a}$ be an ideal of $Q$ that is generated by a regular sequence, and is contained in the monomial ideal $I$. Set $R=Q/\mathfrak{a}$. Any free resolution of $Q/I$ over $Q$ induces a free resolution of $R/I$ over the complete intersection $R$ using a collection of maps (referred to as a \emph{system of higher homotopies}) with respect to $\mathfrak{a}$. This is called the Eisenbud-Shamash construction (see, e.g., \cite{Eisenbud/Peeva:2016}). Moreover, due to the DG-algebra structure, Taylor resolutions and pivot resolutions of $Q/I$ over $Q$ are DG-modules over the Koszul complex on any regular sequence $a_1,\ldots,a_r$ such that $(a_1,\ldots,a_r)\subseteq I$ \cite[Remark 2.2.4]{Avr98}(c.f. \cite[2.5]{Pollitz/Sega:2024}) and thus we can construct a system of higher homotopies on $a_1,\ldots,a_r$ that vanish in degree $2$ and higher (see for instance, \cite[Remark 2.2.1]{Avr98}). An explicit formula for such a system of higher homotopies for Taylor resolutions was recently provided in the work of Sobieska \cite{Sob23}. We provide an analog for pivot resolutions in Theorem \ref{thm:system-pivot}. 

The structure of the paper is as follows. In \Cref{sec:background}, we set up most of the notation and discuss the tools that we use throughout the paper. In \Cref{pivot}, we define pivot complexes corresponding to any increasing sequence which is a subset of $\{1,2,\ldots,q\}$ where $q$ denotes the minimum number of generators of $I$ and prove the necessary and sufficient criterion for this complex to become a resolution by discussing the idea of gaps. We also define the Scarf-number of a monomial ideal $I$ and express the ranks of the free modules appearing in the pivot resolution in terms of these, which in turn give us a criterion to identify when pivot resolutions are minimal (see \Cref{cor:Betti-minimal-pivot}). In particular, in \Cref{thm:q-1-minimal-pivot} we provide a sufficient criterion for the existence of a minimal pivot resolution; however \Cref{ex2} shows that this is not a necessary criterion. We establish the DG-algebra structure of pivot resolutions in \Cref{sec:DG} and finally provide explicit formulae for a system of higher homotopies in \Cref{system}. For the convenience of the reader, we also provide various explicit computations involved in some of the statements from \Cref{system} in the \hyperref[app]{Appendix}.

\section*{Acknowledgements} The second author was partially supported by NSF grants DMS 1801285, 2101671, and 2001368. The third author was supported partially by Project No. 51006801 - American Mathematical Society-Simons Travel Grant.

We would like to thank Aryaman Maithani for some verifications with Sage regarding the heavy computations. We are also grateful to Srikanth Iyengar, Josh Pollitz and Daniel Erman for helpful discussions and for providing various references.

\section{Notations and Preliminaries}\label{sec:background}

%We will use the following notations throughout the paper and  will properly remind readers at the beginning of each section. 

Throughout this article, we will denote $[q]\coloneqq \{1,\dots, q\}$ for any positive integer $q$. We abuse notations by occasionally using $i$ to denote the set $\{i\}$ for any integer $i$. For any $A,B\subseteq [q]$ where the elements are ordered in an increasing sequence, set
    \[
    \sign(A,B)=\begin{cases}
        0 & \text{if } A\cap B\neq \emptyset,\\
        (-1)^{p(A,B)} & \text{if } A\cap B =\emptyset.
    \end{cases}
    \]
where $p(A,B)$ denotes the number of permutations needed to convert the combined sequence $A,B$ into an increasing sequence. 

We will denote $Q$ to be a polynomial ring over a field.  Let $I=(m_1,\dots, m_q)$ be a monomial ideal of $Q$.  For any $A\subseteq [q]$, set $m_A\coloneqq \lcm(\{m_i\colon i\in A\})$.

\subsection{A primer to discrete Morse theory}

% We first recall free resolutions. 
Let $(R,\mathfrak{m})$ be an $\mathbb{N}$-graded ring with the irrelevant maximal ideal $\mathfrak{m}$. Let $M$ be a finitely generated graded $R$-module. A \emph{free resolution} of $M$ is a complex of free $R$-modules
\[
\mathcal{F}\colon \cdots \to R^{n_i}\xrightarrow{\partial_i} R^{n_{i-1}} \to \cdots \to R^{n_1} \xrightarrow{\partial_1} R^{n_0} \to 0
\]
such that $H_0(\mathcal{F}) = M$ and $H_i(\mathcal{F}) = 0$ for any $i>0$. Furthermore, if $\partial(F_i)\subseteq \mathfrak{m} F_{i-1}$ for any $i$, the free resolution $\mathcal{F}$ of $M$ is said to be \emph{minimal}, and then $\beta_i^R(M)\coloneqq n_i$ is called the \emph{$i$-th Betti number} of $M$ over $R$.

% We move on to discrete Morse theory.
Let $\mathcal{P}(X)$ denote the power set of a set $X$. We recall the celebrated Taylor resolution of a monomial ideal $I=(m_1,\dots, m_q)$. Let $\mathcal{T}$ denote the complex
\[
\mathcal{T}: 0\to Q^{\binom{q}{q}} \to \cdots \to Q^{\binom{q}{1}} \to Q^{\binom{q}{0}} \to 0,
\]
where for each $0\leq i\leq q$, let $\{\epsilon_{\tau}\}_{\substack{\tau\in \mathcal{P}([q]),\\ |\tau|=i}}$ denote a basis of $Q^{\binom{q}{i}}$. We note that $Q\epsilon_\tau$ is a free $Q$-module with one generator in multidegree $a_\tau$, which is the exponent vector of $m_\tau$. We define the differentials to be
\[
\partial (\epsilon_{\tau}) = \sum_{\substack{\tau'\subset \tau\\
|\tau'|=|\tau|-1}} [\tau\colon \tau']  \frac{m_\tau}{m_{\tau'}}  \epsilon_{\tau'}
\]
for any $\tau\in \mathcal{P}([q])$. Here by assuming that $\tau=\{i_1,\dots, i_t\}$, where $i_1< \cdots < i_t$, and $\tau'=\tau \setminus \{i_j\}$ for some fixed index $j$, we define $[\tau\colon \tau']=(-1)^{j+1}$. We remark that one can also define the differentials  using the $\sign$ notation:
\[
\partial (\epsilon_{\tau}) = \sum_{j\in \tau} \sign(j,\tau\setminus j)  \frac{m_\tau}{m_{\tau\setminus j }}  \epsilon_{\tau\setminus j}.
\]

The complex $\mathcal{T}$ defined above is a free resolution of $Q/I$, and is more familiarly known as the \emph{Taylor resolution} of $Q/I$ (see  \cite{Tay66}).

We associate the Taylor resolution of $Q/I$ with a directed graph $G=(V,E)$ where
\begin{align*}
    V= \mathcal{P}([q]) \quad \text{and }  \quad 
    E= \{ (\tau\to \tau') \colon \tau \supset \tau' \text{ and } |\tau|=|\tau'|+1 \}.
\end{align*}
\begin{definition}A subset of directed edges $A\subseteq E$ is called a \emph{Morse matching}  if it satisfies the following conditions:
\begin{enumerate}
    \item No two edges in $A$ share a common vertex.
    \item For each edge $(\tau\to \tau')\in A$, we have $m_\tau=m_{\tau'}$.
    \item The directed graph $G^A$ has no directed cycle, where $G^A$ is the directed graph $G$ with edges in $A$ being reversed.
\end{enumerate}
\end{definition}
We note that Morse matchings are also called homogeneous acyclic matchings \cite{BW02} or acylic matchings \cite{Bat02}. We primarily use the following three results.

\begin{lemma}\cite[Lemma 3.2.3]{Bat02}\label{lem:subset-Morse}
    A subset of a Morse matching is a Morse matching.
\end{lemma}

\begin{definition}
    Let $A$ be a Morse matching. Subsets of $\mathcal{P}([q])$ that do not belong to any edge of $A$ are called \emph{$A$-critical}. We will simply say \emph{critical} when $A$ is clear from context.
\end{definition}

For any directed edge $(\tau\to \tau')$ in $G^A$, we define
\[
m(\tau,\tau')=\begin{cases}
    -[\tau'\colon \tau] & \text{if } (\tau'\to \tau)\in A\\
    [\tau\colon \tau'] & \text{otherwise}
\end{cases}
\]
For any \emph{gradient path} $\mathsf{P}$, i.e., a directed path 
\[
\mathsf{P}: \tau_1\to \tau_2\to \cdots \to \tau_t
\]
in the directed graph $G^A$, we define
\[
m(\mathsf{P})=m(\tau_1,\tau_2)\dots m(\tau_{t-1},\tau_t).
\]

Now we are ready to state the key theorem of discrete Morse theory.
\begin{theorem}\protect{\cite[Proposition 2.2, Proposition 3.1, Lemma 7.7]{BW02}}\label{thm:Morse}
    Let $A$ be a Morse matching. Let $\mathcal{F}_A$ be a complex such that for any index $i$, $(\mathcal{F}_A)_i$ is a free $Q$-module with a basis $\{\epsilon_\tau\}$ where $\tau$ ranges among the $A$-critical sets of cardinality $i$, and differentials are defined to be
    \[
    \partial(\epsilon_\tau) = \sum_{\substack{\tau' A\text{-critical}\\
    |\tau'|=|\tau|-1}} \sum_{\substack{\mathsf{P} \text{ gradient path}\\ \text{from } \tau \text{ to } \tau'  }} m(\mathsf{P}) \frac{m_\tau}{m_{\tau'}} \epsilon_{\tau'}.
    \]
    Then $\mathcal{F}_A$ is a free resolution of $Q/I$, called a \emph{Morse resolution} of $Q/I$.
\end{theorem}

In \cite{Ly88}, Lyubeznik considered special subcomplexes of the Taylor resolution which turned out to be free resolutions of $Q/I$. These are now known as \emph{Lyubeznik resolutions}. Later, Batzies and Welker \cite{BW02} showed that these are in fact special cases of Morse resolutions. We will use Batzies and Welker's discovery as the definition of Lyubeznik resolutions.

\begin{theorem}\protect{\cite[Theorem 3.2]{BW02}}\label{thm:Lyubeznik}
    Let $m_{i_1}\succ m_{i_2} \succ \cdots \succ m_{i_q}$ be a total order on $\mingens(I)$. For any $\tau\in \mathcal{P}([q])$, set
    \[
    L(\tau) = \sup \{j\in [q]\colon m_{i_j}\mid \lcm( \{ m_{i_1},\dots, m_{i_{j-1}} \} \cap \tau )  \}. 
    \]
    Then 
    \[
    \{(\tau\cup L(\tau) \to \tau \setminus L(\tau)  )\colon \tau \in \mathcal{P}([q]) \text{ such that } L(\tau)\neq -\infty  \}
    \]
    is a Morse matching. The corresponding Morse resolution is called the \emph{Lyubeznik resolution} of $Q/I$ with respect to $(\succ)$.
\end{theorem}

We remark that $L(\tau)$ depends on the total order $(\succ)$, but in this paper $(\succ)$ is always specified beforehand, so we shorten our notation.

\subsection{The Eisenbud-Shamash construction} \label{ES}
Let $Q$ be a polynomial ring over a field, $\mathfrak{a}$ an ideal of $Q$ generated by a regular sequence $a_1,\dots, a_r$, and $R=Q/\mathfrak{a}$. We recall the Eisenbud-Shamash construction of a free resolution over $R$ from one over $Q$ from \cite{Sha69,Ei80} below, and refer to \cite{Eisenbud/Peeva:2016} and \cite{Avr98} for a more detailed treatment.

Let $\mathcal{F}$ be a complex of free $Q$-modules and $\partial$ denote the differentials. For $a\in \mathbb{Z}$, we denote the shifted complex $\mathcal{F}[-a]$ to be the complex such that $\mathcal{F}[-a]_{i}=\mathcal{F}_{a+i}$. Let $\mathbb{N}$ denote the set of natural numbers, including $\{0\}$. We will use $\mathbf{e}_i$ to denote the $i$-th standard basis vector in $\mathbb{N}^r$ for $i\in [r]$, and $\mathbf{0}$ to denote the zero vector. We denote $|\mathbf{u}|= \sum_{i=1}^r u_i$ for any vector $\mathbf{u}$ in $\mathbb{N}^r$. A \emph{system of higher homotopies} $\sigma$ for $a_1,\dots, a_r$ on $\mathcal{F}$ is a collection of maps
\[
\sigma_{\mathbf{u}}\colon \mathcal{F} \to \mathcal{F}[-2|\mathbf{u}|+1]
\]
that satisfies the following three conditions:
\begin{enumerate}
    \item $\sigma_{\mathbf{0}} = \partial$;
    \item for each $i\in [r]$, the map $\sigma_{\mathbf{0}} \sigma_{\mathbf{e}_i} + \sigma_{\mathbf{e}_i} \sigma_{\mathbf{0}}$ is the multiplication by $a_i$ on $\mathcal{F}$;
    \item if $|\mathbf{u}|\geq 2$, then
    \[
    \sum_{\mathbf{b}+\mathbf{b'}=\mathbf{u}} \sigma_{\mathbf{b}} \sigma_{\mathbf{b'}} = 0.
    \]
\end{enumerate}

We remark that if we set $\sigma_{\mathbf{u}}=0$ for any $|\mathbf{u}|\geq 2$, then for the third condition, it is sufficient to check that $\sigma^2_{\mathbf{e}_j}=0$ for all $j\in [r]$ and 
\[
\sigma_{\mathbf{e}_i} \sigma_{\mathbf{e}_j} + \sigma_{\mathbf{e}_j} \sigma_{\mathbf{e}_i} = 0
\]
for any $i,j\in [r]$ where $i\neq j$.

Next, following the exposition in \cite[Construction 4.1.3]{Eisenbud/Peeva:2016}, we set $\tilde{Q}=Q[t_1,\ldots, t_r]$ with each $t_i$ having degree $-2$. Then $D_{\tilde{Q}}=\operatorname{Hom}_{\text{graded $Q$-modules}}(\tilde{Q},Q)$ is the divided power algebra over $Q$ on the degree $2$ dual variables $y_1,...,y_r$ corresponding to $t_1,...,t_r$ (we refer the reader to \cite[A2]{Eisenbud:2013} for further details on divided power algebras). In other words, we can write $D_{\tilde{Q}}=\oplus Qy_1^{(i_1)}\ldots y_r^{(i_r)}$. The $y_1^{(i_1)}\ldots y_r^{(i_r)}$ are the \textit{divided monomials} that constitute the dual basis to the monomial basis of the polynomial ring $\tilde{Q}$. Note that $D_{\tilde{Q}}$ is a graded module over $\tilde{Q}$ via the action \[t_jy_j^{(i)}=y_j^{(i-1)}.\]Now we assume that $\sigma$ is a system of higher homotopies on $\mathcal{F}$. We set
\[
\Phi (\mathcal{F}) \coloneqq D_{\tilde{Q}}\otimes \mathcal{F} \otimes R
\]
to be a complex of free $R$-modules, with the differential
\[
\delta := \sum t^{\mathbf{u}}\otimes \sigma_{\mathbf{u}}\otimes R
\] where $t^{\mathbf{u}}$ denotes $t_1^{u_1}t_2^{u_2}\ldots t_r^{u_r}$.
This complex $(\Phi(\mathcal{F}),\delta)$ of graded free $R$-modules is called the Eisenbud-Shamash construction. For an explicit example, we refer the reader to \cite[Example 2.6]{Sob23}. The next statement is a celebrated theorem of Eisenbud \cite{Ei80} and Shamash \cite{Sha69} (in the case $r=1$). 

\begin{theorem}\cite[Propositions 3.4.2 and 4.1.4]{Eisenbud/Peeva:2016}\label{thm:ES}
    If $\mathcal{F}$ is a free resolution of a finitely generated $R$-module $N$ over $Q$, then there exists a system of higher homotopies $\sigma$ for the regular sequence $a_1,\ldots, a_r$ and $\Phi(\mathcal{F})$ is a free resolution of $N$ over $R$.
\end{theorem}

As we proceed through the paper, in every section we will set up additional notations and conventions whenever necessary and also make all the hypotheses explicit for convenience.

\section{Pivot complexes and resolutions}\label{pivot}

In this section we will introduce and study pivot resolutions. By construction, the Taylor resolution of $Q/I$ can be associated with the power set $\mathcal{P}([q])$, and any subset $\Omega \subseteq \mathcal{P}([q])$ such that $\Omega$ is closed under taking subsets can be associated to a subcomplex of the Taylor resolution in a canonical way, in which case we denote the corresponding subcomplex by $\mathcal{T}_\Omega$. 

\begin{definition}
    We call $\mathcal{T}_\Omega$ a \emph{pivot complex} of $Q/I$ if either $\Omega = \mathcal{P}([q])$ or there exist $i_1,\dots, i_l$ for some integer $l\geq 2$ such that
    \[
    \Omega=\{\sigma \in \mathcal{P}([q]) \colon \sigma \nsupseteq \{i_1,\dots, i_l\} \}.
    \]
    Thus we will use $\mathcal{T}_{i_1,\dots, i_l}$ to denote this pivot complex.
    Furthermore, we call $\mathcal{T}_{i_1,\dots, i_l}$ a \emph{pivot resolution} of $Q/I$ if it is additionally a resolution.
\end{definition}

\begin{example}
    Let $Q=\mathbb{Q}[w,x,y,z]$ and $I=(wx,xy,yz)$. Set $m_1\coloneqq wx, m_2\coloneqq xy$ and $m_3\coloneqq yz$. The Taylor resolution $\mathcal{T}$ and pivot complex $\mathcal{T}_{1,2}$ of $Q/I$ are
    \[
    \mathcal{T}: 0\to Q\epsilon_{123} \xrightarrow{\begin{pmatrix}
        w\\
        -1\\
        z
    \end{pmatrix}} \begin{matrix}
        Q\epsilon_{23}\\
        \oplus\\
        Q\epsilon_{13}\\
        \oplus\\
        Q\epsilon_{12}
    \end{matrix} \xrightarrow{\begin{pmatrix}
        0&yz&y\\
        z&0&-w\\
        -x&-wx&0
    \end{pmatrix}} \begin{matrix}
        Q\epsilon_{1}\\
        \oplus\\
        Q\epsilon_{2}\\
        \oplus\\
        Q\epsilon_{3}
    \end{matrix}
    \xrightarrow{\begin{pmatrix}
        wx & xy & yz
    \end{pmatrix}} Q\epsilon_{\emptyset}\to 0
    \]
    and
    \[
    \mathcal{T}_{1,2}: 0\to \begin{matrix}
        Q\epsilon_{23}\\
        \oplus\\
        Q\epsilon_{13}\\
    \end{matrix} \xrightarrow{\begin{pmatrix}
        0&yz\\
        z&0\\
        -x&-wx
    \end{pmatrix}} \begin{matrix}
        Q\epsilon_{1}\\
        \oplus\\
        Q\epsilon_{2}\\
        \oplus\\
        Q\epsilon_{3}
    \end{matrix}
    \xrightarrow{\begin{pmatrix}
        wx & xy & yz
    \end{pmatrix}} Q\epsilon_{\emptyset}\to 0,
    \]
    respectively. We observe that $\begin{pmatrix}
        y\\-x\\0
    \end{pmatrix}$ is in $\ker \partial_1$, but not in $\image \partial_2$. In other words, $H_1(\mathcal{T}_{1,2})\neq 0$ and in particular, $\mathcal{T}_{1,2}$ is not a resolution of $Q/I$.
\end{example}

Given an index set $\tau\subseteq [q]$ and $h\not \in \tau$  with $h\in [q]$, we say that $h$ is a \emph{gap} of $\tau$ if $m_h\mid  m_{\tau}$. It is clear that $h$ is a gap of $\tau$ if and only if $m_\tau = m_{\tau\cup h}$. Using this notion, we determine when pivot complexes are resolutions in the following theorem.

\begin{theorem}\label{thm:pivot-when-resolution}
    The pivot complex $\mathcal{T}_{i_1,\dots, i_l}$, where
    $l\geq 2$ and $1\leq i_1<\cdots < i_l\leq q$, is a resolution if and only if $\{i_1,\dots, i_l\}$ has a gap.
\end{theorem}
\begin{proof}
    Up to relabeling, we can assume that $i_j=j$ for any $1\leq j\leq l$. In particular, this means $\{i_1,\dots, i_l\}=[l]$.
    
    Assume that the pivot complex $\mathcal{T}_{1,\dots, l}$ is a resolution. Since $\mathcal{T}_{1,\dots, l}$ is a subcomplex of the Taylor resolution $\mathcal{T}$, by following the proof of \cite[Theorem 1.6]{Roberts:1980}, up to isomorphism, $\mathcal{T}_{1,\dots, l}$ can be obtained from $\mathcal{T}$ after ``removing" summands of the form
    \[
    0\to Q\epsilon_{\tau\cup h} \xrightarrow{1} Q\epsilon_{\tau}\to 0
    \]
     for some $\tau\in \mathcal{P}([q])$ and $h\notin \tau$ such that these two basis elements have the same multidegree $a_{\tau}=a_{\tau \cup h}$, and hence $m_{\tau}=m_{\tau \cup h}$. By definition of $\mathcal{T}_{1,\dots, l}$, the copy $Q\epsilon_{[l]}$ of the Taylor resolution no longer exists in $\mathcal{T}_{1,\dots, l}$, and all copies of $Q$ in homological degree $l-1$ remain in $\mathcal{T}_{1,\dots, l}$. Thus there exists a $\tau\in \mathcal{P}([q])$ such that $|\tau|=l+1$ and $m_{[l]}=m_{\tau}$. Then any $h\in \tau \setminus [l]$ satisfies our required criterion.

    Conversely, assume that $h\in [q]$ is a gap of $[l]$. Set
    \[
    A\coloneqq \{\tau\cup h\to \tau \setminus h \colon \tau\supseteq [l] \}.
    \]
    Consider a total order on $\mingens(I)$ such that $m_i\succ m_h$ for any index $i\neq h$. Then it is straightforward that $L(\tau)=h$ for any $\tau \supseteq [l]$. Thus $A$ is a subset of the Morse matching that induces a Lyubeznik resolution in Theorem \ref{thm:Lyubeznik}, and therefore is a Morse matching itself by Lemma \ref{lem:subset-Morse}. The basis elements of the Morse resolution $\mathcal{F}_A$ as in \cref{thm:Morse}, can be identified with basis elements of the pivot complex $\mathcal{T}_{1,\dots, l}$ canonically. Finally, since the set of $A$-critical sets in this case is closed under taking subsets, the differentials of $\mathcal{F}_A$ are exactly the restrictions of those of the Taylor resolution $\mathcal{T}$ by \cite[Proposition 5.2]{CK24}. Thus, $\mathcal{F}_A$ and $\mathcal{T}_{1,\dots, l}$ are the same. In particular, $\mathcal{T}_{1,\dots, l}$ is a free resolution of $Q/I$, as required.
\end{proof}

\begin{remark}\label{rem:Lyubeznik-pivot-Taylor}
    In the proof above, we showed that any pivot resolution is induced from a Morse matching that is a subset of the Morse matching that induces a Lyubeznik resolution. Since their differentials are restrictions of the Taylor differentials, for any pivot resolution, we can always find a Lyubeznik resolution that is its subcomplex canonically. Thus
    \[
    \text{Lyubeznik resolution } \subseteq \text{Pivot resolution } \subseteq \text{Taylor resolution}
    \]
    as complexes, and the embeddings are canonical.
\end{remark}

\begin{example}
    Let $I=(x_1^2, x_2^2, x_3^2, x_1x_2x_3)$. Set $m_1, m_2, m_3,$ and $m_4$ to be the four generators accordingly. Then
    \[
    \mathcal{T}_{1,2,3}: 0\to Q^3 \to Q^6 \to Q^4\to Q
    \]
    is a resolution of $Q/I$ since $4$ is a gap of $\{1,2,3\}$. On the other hand, 
    \[
    \mathcal{T}_{1,2}: 0\to Q^2 \to Q^5 \to Q^4\to Q
    \]
    is not a resolution of $Q/I$ since $\{1,2\}$ has no gap.  
\end{example}

By definition, the Taylor resolution is a pivot resolution. In general, this may be the only pivot resolution. In fact, by the preceding theorem, it is clear that this happens exactly when the Taylor resolution is minimal. This means that barring the aforementioned case, there is always a pivot resolution that is shorter than the Taylor resolution. It is natural to ask which pivot resolution is the ``smallest" for a fixed monomial ideal $I$. For this purpose, we introduce a new notion.

\begin{definition}\label{def:scarf-number}
    Let $I$ be a monomial ideal of $S$. We define the \emph{Scarf number} of $I$, denoted by $\scarfnumber(I)$, to be 
    \[
    \inf\{t\in \mathbb{N}\colon \exists \tau, \tau' \in \mathcal{P}([q]) \text{ such that } \tau\neq \tau', \  |\tau|=t \text{ and }  m_\tau= m_{\tau'}   \}.
    \]
\end{definition}

A \emph{Scarf index set} of $I$ is a subset $\tau\subseteq [q]$ such that $m_\tau=m_{\tau'}$ for some $\tau'\in \mathcal{P}([q])$ implies $\tau=\tau'$. Thus the Scarf number of $I$ is simply the smallest possible cardinality of a non-Scarf index set of $I$. This definition is inspired by the Scarf complex introduced in \cite{BPS98}. From the definition, it is easy to see that either $2\leq \scarfnumber(I)\leq \mu(I)-1$ or $\scarfnumber(I)=\infty$, and the Taylor resolution of $Q/I$ is minimal if and only if $\scarfnumber(I)=\infty$. Furthermore, in this case, by definition, there is no other pivot resolutions. 

\begin{example}\label{exa:scarf-number}
    Let $Q=\mathbb{Q}[u,w,x,y,z]$ and $I=(u,wx,xy,yz)$. We will show that $\scarfnumber(I)=2$. We already have $\scarfnumber(I)\geq 2$. Thus it suffices to find an index set of $I$ that has a gap and is of cardinality 2. Set $m_1=u, m_2=wx, m_3=xy, m_4=yz$. Then $3$ is a gap of $\{2,4\}$, as required. 
\end{example}

A direct corollary of our preceding theorem is we can determine the ``smallest" pivot resolution of $Q/I$.  We use the convention that $\binom{-\infty}{-\infty} = 0$.

\begin{corollary}
    Set $l=\scarfnumber(I)$ and assume that $l\neq \infty$. Then there exist indices $i_{1},\dots, i_l$ such that 
    $\mathcal{T}_{i_{1},\dots, i_l}$ is a pivot resolution. Moreover, for any pivot resolution $\mathcal{F}$, we have
    \[
    \binom{q}{i} - \binom{q-\scarfnumber(I)}{i-\scarfnumber(I)}=\rank (\mathcal{T}_{i_1,\dots, i_l})_i \leq \rank (\mathcal{F})_i=\binom{q}{i}
    \]
    for any index $i$.
\end{corollary}
\begin{proof}
    It is easy to see that any index set that has a gap is non-Scarf, and any minimal non-Scarf index set has a gap. By the definition of $l$, there exist indices $i_{1},\dots, i_l$ such that the set of these indices has a gap. By \cref{thm:pivot-when-resolution}, $\mathcal{T}_{i_{1},\dots, i_l}$ is a pivot resolution. 

    Also by \cref{thm:pivot-when-resolution}, the ``smallest" pivot resolution comes from an index set that has a gap with the smallest cardinality, which, in this case, is $\{i_1,\dots, i_l\}$. Finally, we compute the rank of $(\mathcal{T}_{i_1,\dots, i_l})_i$ for any index $i$. By the construction of this pivot resolution, we have
    \begin{align*}
        \rank (\mathcal{T}_{i_1,\dots, i_l})_i &= \rank (\mathcal{T})_i - |\{ \tau \in \mathcal{P}([q]) \colon \tau \supseteq \{i_1,\dots, i_l\} \text{ and } |\tau|=i \}|\\
        &= \binom{q}{i} - \binom{q-l}{i-l}= \binom{q}{i} - \binom{q-\scarfnumber(I)}{i-\scarfnumber(I)},
    \end{align*}
    as required.
\end{proof}

This pivot resolution gives a new bound on Betti numbers for any monomial ideal. 

\begin{corollary}\label{cor:Betti-minimal-pivot}
    For any $i\in \mathbb{Z}$, we have
    \[
    \beta_i^Q(Q/I) \leq \binom{q}{i} - \binom{q-\scarfnumber(I)}{i-\scarfnumber(I)}.
    \]
    Moreover, the equality occurs for any integer $i$ if and only if $Q/I$ has a minimal pivot resolution.
\end{corollary}

Before moving on, we note that $Q/I$ has a minimal pivot resolution in another special case.

\begin{theorem}\label{thm:q-1-minimal-pivot}
    If  $\scarfnumber(I)\geq q-1$, then $Q/I$ has a minimal pivot resolution. In particular, if $q\leq 3$, then $Q/I$ has a minimal pivot resolution.
\end{theorem}
\begin{proof}
    If $\scarfnumber(I)=\infty$, then the Taylor resolution is minimal, as required. We can now assume that $\scarfnumber(I)=q-1$.
    For any index $i$, we have
    \[
    \binom{q}{i}\geq \beta_i^Q(Q/I)\geq |\{\text{Scarf index sets of cardinality } i \}|,
    \]
    where the first inequality comes from the ranks of the free modules of the Taylor resolution, and the second is by \cite[Theorem 59.2]{Pe11}. Since $\scarfnumber(I)=q-1$, any index set of cardinality at most $q-2$ is Scarf. Thus we have $\beta_i^Q(Q/I) = \binom{q}{i}$  for any $i\leq q-2$ due to the above inequalities. Also since $\scarfnumber(I)=q-1$, we have $\beta_q^Q(Q/I)=0$ by \cref{cor:Betti-minimal-pivot}. Since the alternating sum of total Betti numbers of $Q/I$ equals 0, we easily obtain $\beta_{q-1}^Q(Q/I)=q-1=\binom{q}{q-1}-\binom{q-(q-1)}{(q-1)-(q-1)}$. By the preceding result, $Q/I$ has a minimal pivot resolution.

    When $q\leq 3$, we have $2\leq \scarfnumber(I)\leq q-1\leq 2$. In other words, $\scarfnumber(I)=q-1$, thus the result follows. 
\end{proof}

If $\scarfnumber(I)<q-1$, it is inconclusive whether $Q/I$ has a minimal pivot resolution. We illustrate this in the following discussion. 

\begin{example}\label{ex2}
    Set $Q=\mathbb{Q}[u,w,x,y,z]$, $I_1=(wx, xy, yz, wz)$ and $I_2=(u,wx,xy,yz)$. We recall that $I_2$ is the ideal in \cref{exa:scarf-number}. It is clear that $\mu(I_1)=\mu(I_2)=4$ and following \cref{exa:scarf-number}, one can verify that $\scarfnumber(I_1)=\scarfnumber(I_2)=2$. 
    
    Using Macaulay2 \cite{M2}, one can check the Betti numbers of $Q/I_1$ are $(1,4,4,1)$ and those of $Q/I_2$ are $(1,4,5,2)$. By \cref{cor:Betti-minimal-pivot}, $Q/I_2$ has a minimal pivot resolution, but $Q/I_1$ does not.
\end{example}

\section{Differential Graded Algebra Structure}\label{sec:DG}

We recall the definition of DG algebras. Note that we assume that a DG algebra is graded-commutative and associative.

\begin{definition}
    Let $S$ be a ring, $(F,\partial)$ a complex of $S$-modules where $F_i=0$ for any $i<0$. If we have an element $a\in F_n$ for some $n$, then let $|a|$ denote the homological degree of $a$, i.e., $|a|=n$. Then $(F,\partial)$ is a \emph{DG-algebra} if it is equipped  with a multiplication structure $(\star)$ and a unit element $1\in F_0$ and for any homogeneous $a,b\in F$, we have
    \begin{enumerate}
        \item (unitary) $a\star 1 = 1 \star a =a$.
        \item (graded-commutativity) $a\star b=(-1)^{|a||b|}b\star a$, and if $|a|$ is odd, $a\star a =0$.
        \item (associativity) $(a\star b) \star c=a\star (b\star c)$.
        \item (Leibniz's rule) $\partial(a\star b)=\partial(a)\star b+(-1)^{|a|} a\star \partial(b)$.
    \end{enumerate}
\end{definition}

Taylor resolutions are famous examples of DG-algebras. Recall that we use $\{\epsilon_A \}_{A\in \mathcal{P}([q])}$ to denote the bases for the free modules in a Taylor resolution $\mathcal{T}$. Gemeda \cite{Ge76} showed that Taylor resolutions are DG-algebras with the following multiplication:
\[
    \epsilon_A\star \epsilon_B=\begin{cases}
        0& \text{if }A\cap B\neq \emptyset,\\
        \sign(A,B) \frac{m_A m_B}{m_{A\cup B}} \epsilon_{A\cup B} &  \text{if } A\cap B = \emptyset.
    \end{cases} 
\]
where $A,B\in \mathcal{P}([q])$. We will show that all pivot complexes are also DG-algebras.  

\begin{theorem}\label{thm:pivot-is-DG}
     Any pivot resolution of $Q/I$ over $Q$ is a DG-algebra.
\end{theorem}
\begin{proof}
    By relabelling, we can assume that any pivot complex is of the form $\mathcal{T}_{1,\dots, l}$ for some integer $l$. Assume that $\mathcal{T}_{1,\dots, l}$ is a resolution of $Q/I$, i.e., by the proof of \cref{thm:pivot-when-resolution}, $\mathcal{T}_{1,\dots, l}$ is a Morse resolution induced by the Morse matching
    \[
    A=\{ \tau\cup h \to \tau \setminus h \colon \tau \supseteq [l] \},
    \]
    for some $h>l$. By discrete Morse theory (see, e.g., \cite[Proof of Theorem 4.1]{Kat19}), $\mathcal{T}_{1,\dots, l}$ is exactly $\mathcal{T}/\mathcal{I}$ where \[
    \mathcal{I}\coloneqq \operatorname{span}_Q \{ \epsilon_\tau, \partial\epsilon_\tau \colon \tau \supseteq [l] \cup h  \}.  \]
    By \cite[Lemma 10.36]{BSW}, it suffices to show that $\mathcal{I}$
    is a DG-ideal of $\mathcal{T}$, i.e., we will show three things:
    \begin{enumerate}
        \item $\mathcal{I}$ satisfies the Leibniz rule: This is automatic since it is a subset of $\mathcal{T}$.
        \item $\mathcal{I}$ is closed under taking the differential: Indeed, for any $\tau \supseteq [l] \cup h$, the elements $\partial (\epsilon_\tau)$ and $
        \partial(\partial \epsilon_\tau)=0$ 
        are in $\mathcal{I}$.
        \item $\mathcal{I}$ is an ideal of $\mathcal{T}$: It suffices to show that for any $\tau \supseteq [l] \cup h$ and any $\sigma \subseteq [q]$, the two elements  $\epsilon_\tau \star \epsilon_\sigma$ and $\partial(\epsilon_\tau ) \star  \epsilon_\sigma$ are in $\mathcal{I}$. We will show $\epsilon_\tau \star  \epsilon_\sigma\in \mathcal{I}$ first. Indeed, if $\tau\cap \sigma \neq \emptyset$, then $\epsilon_\tau \star \epsilon_\sigma=0\in \mathcal{I}$. On the other hand, if $\tau\cap \sigma = \emptyset$, then
        \[
        \epsilon_\tau \star \epsilon_\sigma = \sign(\tau,\sigma) \frac{m_\tau m_\sigma}{m_{\tau\cup\sigma}} \epsilon_{\tau\cup\sigma},
        \]
        and since $\tau \cup \sigma \supseteq [l] \cup h$, we have $\epsilon_\tau \star \epsilon_\sigma\in \mathcal{I}$, as required.

        Finally, we show that $\partial(\epsilon_\tau ) \star  \epsilon_\sigma \in \mathcal{I}$. Indeed, by Leibniz rule, we have
        \[
        \partial(\epsilon_\tau ) \star  \epsilon_\sigma = \partial(\epsilon_\tau \star \epsilon_\sigma ) - (-1)^{|\tau|} \epsilon_\tau \star \partial \epsilon_\sigma \in \mathcal{I}
        \]
        since the product of $\epsilon_\tau$ and any other basis element is in $\mathcal{I}$, as required.\qedhere
    \end{enumerate}
\end{proof}

\begin{remark}
    The proof above hinges on the fact that when pivot complexes are resolutions, they are Morse resolutions, and hence in particular, already quotients of the Taylor resolution. As a matter of fact, one can show that pivot complexes are DG-algebras directly by  defining the multiplication $(\star)$ on $\mathcal{T}_{1,\dots, l}$ to be
\[
    \epsilon_A\star \epsilon_B=\begin{cases}
        0& \text{if }A\cap B\neq \emptyset \text{ or } [l+1] \subseteq A\cup B,\\
        \sign(A,B) \frac{m_A m_B}{m_{A\cup B}} \epsilon_{A\cup B} &  \text{if } A\cap B = \emptyset \text{ and } [l] \nsubseteq A\cup B,
        \\
        \sign(A,B)\sum_{i\in [l]} \sign(i,A\cup B\setminus i) \displaystyle\frac{m_A m_B}{m_{A\cup B\cup (l+1) \setminus i}} \epsilon_{A\cup B\cup (l+1) \setminus i} & \text{if } A\cap B = \emptyset,\ [l] \subseteq A\cup B, \text{ and } l+1 \notin A\cup B,
    \end{cases} 
\]
    and verify the necessary conditions. We leave this part to interested readers.
    
    We call these pivot complexes because of this multiplication rule. Basically,  $(l+1)$ appears whenever the multiplication rule for the Taylor resolution no longer makes sense in the pivot resolutions. We call $l+1$ a ``pivot" that we use to modify the formulae and fit the rules.
\end{remark}

\begin{remark}
    We have one immediate application of the fact that pivot resolutions are canonically quotient DG-algebras of the Taylor resolution. Assume $I$ and $J$ are monomial ideals of $Q$ such that $I\subseteq J$. Then we have the following maps of DG-algebras by Lemma \cite[Lemma 4.2]{Iyen97}:
    \[
    \text{Taylor resolution of $Q/I$} \to \text{Taylor resolution of $Q/J$} \to \text{pivot resolution of $Q/J$}.
    \]
    By Theorem \cite[Theorem 1.2]{Iyen97}, we obtain a free resolution of $Q/J$ over $Q/I$ using a pivot resolution of $Q/J$ and the Taylor resolution of $Q/I$, thus obtaining an improvement of \cite[Theorem 4.3]{Iyen97}.
\end{remark}

\section{A system of higher homotopies for pivot resolutions}\label{system}

 Let $Q$ be a polynomial ring over a field and  $\mathfrak{a}$ be an ideal of $Q$ generated by a regular sequence $a_1,\dots, a_r$, and $R=Q/\mathfrak{a}$. Let $I=(m_1,\dots, m_q)$ be a monomial ideal of $Q$, with $l\coloneqq \scarfnumber(I)$. We assume that the Taylor resolution of $Q/I$ is not minimal, i.e., $l<\infty$, and $l+1$ is a gap of $[l]$.  We further assume that $\mathfrak{a}\subseteq I$. Thus we can write $a_i=a_{i1}m_1+\cdots + a_{iq}m_q$ for $i\in [r]$ and $a_{ij}\in Q$ where $j\in [q]$.
  By the Eisenbud-Shamash construction discussed in \Cref{ES}, a pivot resolution can be lifted to a resolution of $R/I$ over $R$ using a system of higher homotopies (\cref{thm:ES}). In this section, our goal is to explicitly describe such a system for pivot resolutions.
 
 Without loss of generality, we assume that $\mathcal{T}_{1,\dots, l}$ is a pivot resolution of $Q/I$.  We give explicit formulae for a system of higher homotopies for $a_1,\dots, a_r$ on $\mathcal{T}_{1,\dots, l}$. 

\begin{theorem}\label{thm:system-pivot}
    For any $s\in [r]$, set
\begin{align*}
        \sigma_{\mathbf{e}_s}(\epsilon_A)=\begin{cases}
            \sum_{j\notin A} \sign(j,A) a_{sj}\frac{m_jm_A}{m_{A\cup j}}  \epsilon_{A\cup j} & \text{if } |A\cap [l]|\neq l-1\\
            \sum_{j\notin A\cup t} \sign(j,A) a_{sj}\frac{m_jm_A}{m_{A\cup j}}  \epsilon_{A\cup j} & \text{if } [l] \setminus A = \{t\} \text{ and } l+1\in A \\
            \sum_{j\notin A\cup t} \sign(j,A) a_{sj}\frac{m_jm_A}{m_{A\cup j}}  \epsilon_{A\cup j}\\ +(-1)^{l+1} \sign(t,A)a_{st} \sum_{i\in [l]} \sign(i,A\cup t\setminus i) \frac{m_tm_A}{m_{A\cup t\cup (l+1)\setminus i}}\epsilon_{A\cup t\cup  (l+1)\setminus i} & \text{if } [l] \setminus A = \{t\} \text{ and } l+1\notin A
        \end{cases}
    \end{align*}
    and
\[
\sigma_{\mathbf{0}}=\delta, \qquad \sigma_{\mathbf{u}}=0 \text{ if } |\mathbf{u}|\geq 2.
\]
where $\delta$ denotes the differential in $\mathcal{T}_{1,\ldots,l}$.    Let $(\sigma)$ denote this system. Then
    $(\sigma)$ is a system of higher homotopies for $a_1,\dots, a_r$ on the pivot resolution $\mathcal{T}_{1,\dots, l}$.
\end{theorem}

Notice that for Taylor resolutions, the first case of this formula matches with \cite[Definition 3.1]{Sob23}. Thus as long as we only use the first case, many of our results follow from \cite{Sob23}. We break the proof of \Cref{thm:system-pivot} into several lemmas. We start by recalling some needed identities regarding $\sign$.

\begin{lemma} \label{lemma2.1} The following identities hold. 
    \begin{enumerate}
        \item For any $A\subseteq [q]$ and $i\notin A$, we have
        \[
        \sign(i,A)=(-1)^{|\{ j\in [q]\colon j<i \}\cap A|}.
        \]
        \item For any $B,C\subseteq [q]$, we have
        \[
        \sign(B,C)=(-1)^{|B||C|}\sign(C,B).
        \] \label{lemma2.1item1}
        \item For any $A,B,C\subseteq [q]$, we have
        \[
        \sign(A,B)\sign(A\cup B,C)=\sign(B,C) \sign(A,B\cup C).
        \]
        \item For any $A\subseteq [q]$ and any $i,j\in C$ where $i\neq j$, we have
        \begin{align*}
            \sign(i,A\setminus i)\sign(j,A\setminus i\setminus j)+\sign(j,A\setminus j)\sign(i,A\setminus i\setminus j) &= 0,\\
            \sign(i,A\setminus i\setminus j)\sign(j,A\setminus i\setminus j)+\sign(i,A\setminus i)\sign(j,A\setminus j) &= 0
        \end{align*}
    \end{enumerate}
\end{lemma}
\begin{proof}
    \begin{enumerate}
        \item This follows from the definition: The amount of permutations needed to turn $i,A$ into an increasing sequence is exactly the number of indices in $A$ that is less that $i$.
        \item Set
    \begin{align*}
        B=\{b_1,\dots, b_r\},\quad 
        C=\{c_1,\dots, c_s\},    \end{align*}
    where $b_1<\cdots <b_r$ and $c_1<\cdots < c_s$ for some integer $s,t$. If $B\cap C\neq \emptyset$, then both $\sign(B,C)$ and $\sign(C,B)$ equal $0$, and thus the result follows. Now we can assume that $B\cap C =\emptyset$. To prove the statement, it suffices to show that it takes $|B||C|$ transpositions needed to turn the sequence
    \[
    b_1,\dots, b_r, c_1,\dots, c_s
    \]
    into 
    \[
    c_1,\dots, c_s, b_1,\dots, b_r.
    \]
    Indeed, to obtain this feat, we first move $b_r$ to the end of the sequence, which takes $|C|$ transpositions. Next, we move $b_{r-1}$ to right before $b_r$, which takes another $|C|$ transpositions. Repeating this argument, the number of transpositions needed is exactly $\sum_{i=1}^r |C|=r|C|=|B||C|$, as required.
        \item Both sides are equal to $(-1)^r$ where $r$ is the amount of transpositions needed to turn the sequence $A,B,C$ into an increasing sequence. Indeed, the left-hand side turns $A,B$ into an increasing sequence first, and then $A\cup B, C$, while the right-hand side turns $B,C$ first, and $A,B\cup C$ later.
        \item  Without loss of generality, we assume that $i<j$. By part (1), we have $\sign(i,A\setminus i\setminus j) = \sign(i,A\setminus i)$ and $\sign(j,A\setminus i\setminus j) = -\sign(j,A\setminus j)$. Thus
        \begin{align*}
            &\ \ \ \ \sign(i,A\setminus i) \sign(j,A\setminus i \setminus j)+ \sign(j,A\setminus j) \sign(i,A\setminus i \setminus j) \\
            &= -\sign(i,A\setminus i)\sign(j,A\setminus j) + \sign(j,A\setminus j)\sign(i,A\setminus i) = 0
        \end{align*}
        and 
        \begin{align*}
            &\ \ \ \ \sign(i,A\setminus i\setminus j) \sign(j,A\setminus i \setminus j)+ \sign(i,A\setminus i) \sign(j,A\setminus j) \\
            &= -\sign(i,A\setminus i)\sign(j,A\setminus j) + \sign(i,A\setminus i)\sign(j,A\setminus j) = 0,
        \end{align*}
        as required. \qedhere
    \end{enumerate}
\end{proof}

We now start checking the conditions needed for $(\sigma)$ to be a system of higher homotopies of $a_1,\dots, a_r$ on $\mathcal{T}_{1,\dots, l}$.

\begin{lemma}\label{lem:Ahasl+1del}
    For any $1\leq s\leq r$ and any $A \nsupseteq [l]$ such that $[l]\setminus A=\{t\}$ and $l+1\in A$, we have 
    \[(\partial\sigma_{\mathbf{e}_s}+\sigma_{\mathbf{e}_s}\partial)(\epsilon_A)=a_s \epsilon_A.
    \]
\end{lemma}

\begin{proof}
    We have
    \begin{align*}
        &\ \ \ \ \ (\partial\sigma_{\mathbf{e}_s}+\sigma_{\mathbf{e}_s}\partial)(\epsilon_A)\\
        &= \partial\big(\sum_{j\notin A\cup t} \sign(j,A) a_{sj}\frac{m_jm_A}{m_{A\cup j}}  \epsilon_{A\cup j}\big) + \sigma_{\mathbf{e}_s} \big( \sum_{i\in A} \sign(i,A\setminus i) 
        \frac{m_A}{m_{A\setminus i}} \epsilon_{A\setminus i} \big)\\
        &=\!\begin{multlined}[t][.5\displaywidth] 
        \partial\big(\sum_{j\notin A\cup t} \sign(j,A) a_{sj}\frac{m_jm_A}{m_{A\cup j}}  \epsilon_{A\cup j}\big) + \sigma_{\mathbf{e}_s} \big( \sum_{i\in A\cap [l]} \sign(i,A\setminus i)\frac{m_A}{m_{A\setminus  i}} \epsilon_{A\setminus i} \big)\\
        + \sigma_{\mathbf{e}_s} \big( (-1)^{l+1} \frac{m_A}{m_{A\setminus (l+1)}} \epsilon_{A\setminus (l+1)} \big) + \sigma_{\mathbf{e}_s} \big( \sum_{i\in A\setminus [l+1]} \sign(i,A\setminus i) \frac{m_A}{m_{A\setminus i}}\epsilon_{A\setminus i} \big)
        \end{multlined}\\
        &=\!\begin{multlined}[t][.5\displaywidth]
			   \sum_{j\notin A\cup t} \sign(j,A) a_{sj}  \sum_{i\in A\cup j} \sign(i,A\cup j\setminus i) \frac{m_jm_A}{m_{A\cup j\setminus i}}  \epsilon_{A\cup j\setminus i} +  \sum_{i\in A\cap [l]} \sign(i,A\setminus i) \sum_{j\notin A\setminus i} \sign(j,A\setminus i) a_{sj} \frac{m_jm_A}{m_{A\cup j\setminus i}}  \epsilon_{A\cup j\setminus i} \\
            + (-1)^{l+1}\big( \sum_{j\notin A\cup t\setminus (l+1)} \sign(j,A\setminus (l+1)) a_{sj}\frac{m_jm_A}{m_{A\cup j \setminus (l+1)}}   \epsilon_{A\cup j\setminus (l+1)} \\+(-1)^{l+1} \sign(t,A\setminus (l+1))a_{st} \sum_{i\in [l]} \sign(i,A\cup t\setminus (l+1)\setminus i) \frac{m_t m_A}{m_{A\cup t \setminus i}}  \epsilon_{A\cup t\setminus i} \big) \\
            +  \sum_{i\in A\setminus [l+1]} \sign(i,A\setminus i) \sum_{j\notin A\cup t\setminus i} \sign(j,A\setminus i) a_{sj}\frac{m_j m_A}{m_{A\cup j \setminus i}}  \epsilon_{A\cup j\setminus i}      \end{multlined}.
    \end{align*}
    We will simplify the coefficient for each basis element in the above sum.
    %\begin{itemize}
     We start by looking at the coefficient for $\epsilon_A$:
        \begin{align*}
            &\!\begin{multlined}[t][.5\displaywidth]
			   \sum_{j\notin A\cup t} \sign(j,A)a_{sj}m_j \sign(j,A) + \sum_{i\in A\cap [l]} \sign(i,A\setminus i) \sign(i,A\setminus i) a_{si}m_i \\
                + (-1)^{l+1}\big( \sign(l+1,A\setminus(l+1))a_{s,l+1}m_{l+1}
                + (-1)^{l+1} \sign(t,A\setminus (l+1)) a_{st}m_t \sign(t,A\setminus (l+1)) \big) \\
                + \sum_{i\in A\setminus [l+1]} \sign(i,A\setminus i) \sign(i,A\setminus i) a_{si}m_i
            \end{multlined}\\
            &=\sum_{j\notin A\cup t} a_{sj}m_j  + \sum_{i\in A\cap [l]} a_{si}m_i + (-1)^{l+1}\big( (-1)^{l-1} a_{s,l+1}m_{l+1} 
                + (-1)^{l+1} a_{st}m_t \big) + \sum_{i\in A\setminus [l+1]}  a_{si}m_i,            
        \end{align*}
        where in the first equality we used \cref{lemma2.1} (1). We now rearrange the coefficients:
        \begin{align*}
            &\ \ \ \ \sum_{j\notin A\cup t} a_{sj}m_j  + \sum_{i\in A\cap [l]} a_{si}m_i +  a_{s,l+1}m_{l+1} 
                + a_{st}m_t + \sum_{i\in A\setminus [l+1]}  a_{si}m_i\\
            &=\big( \sum_{i\in A\cap [l]} a_{si}m_i +   a_{s,l+1}m_{l+1} 
                + a_{st}m_t \big) +\big( \sum_{j\notin A\cup t} a_{sj}m_j  +\sum_{i\in A\setminus [l+1]}  a_{si}m_i\big)\\
            &=\big( \sum_{i\in [l],i\neq t} a_{si}m_i +   a_{s,l+1}m_{l+1} 
                + a_{st}m_t \big) +\big( \sum_{j\notin A, j>l+1} a_{sj}m_j  +\sum_{i\in A, i>l+1}  a_{si}m_i\big)\\
            &=\big( \sum_{i\in [l+1]} a_{si}m_i \big) +\big( \sum_{j>l+1}a_{sj}m_j  \big)\\
            &=\sum_{i\in [q]} a_{si}m_i \\
            &= a_s.            
        \end{align*}
         The coefficient of $\epsilon_{A\cup t\setminus i}$ where $i\in [l]\setminus t$ is:
        \begin{align*}
            & \sign(i,A\setminus i) \sign(t,A\setminus i) a_{st}\frac{m_t m_A}{m_{A\cup t\setminus i}}  + (-1)^{l+1}\big( (-1)^{l+1} \sign(t,A\setminus (l+1)) a_{st}\frac{m_t m_A}{m_{A\cup t\setminus i}} \sign(i,A\cup t\setminus(l+1)\setminus i) \big)\\
            &=\sign(i,A\setminus i) \sign(t,A\setminus i) a_{st}\frac{m_t m_A}{m_{A\cup t\setminus i}} +   \sign(t,A)  \sign(i,A\cup t\setminus i) a_{st}\frac{m_t m_A}{m_{A\cup t\setminus i}}\\
            &=0,
        \end{align*}where the last equality is from the second equality of \Cref{lemma2.1}(4).
         The coefficient of $\epsilon_{A\cup j\setminus (l+1)}$ where $j\notin A \cup t \cup (l+1)$ is:
        \begin{align*}
            &\ \ \ \  \sign(j,A)a_{sj}\frac{m_j m_A}{m_{A\cup j\setminus (l+1)}}\sign(l+1,A\cup j\setminus(l+1))+(-1)^{l+1}(\sign(j,A\setminus(l+1))a_{sj}\frac{m_j m_A}{m_{A\cup j\setminus (l+1)}})
             \\
            &= \sign(j,A)a_{sj}\frac{m_j m_A}{m_{A\cup j\setminus (l+1)}}(-1)^{l+1}+(-1)^{l+1}((-1)\sign(j,A)a_{sj}\frac{m_j m_A}{m_{A\cup j\setminus (l+1)}})\\
            &=(-1)^{l+1} \sign(j,A)a_{sj}\frac{m_j m_A}{m_{A\cup j\setminus (l+1)}}+(-1)^{l+2}\sign(j,A)a_{sj}\frac{m_j m_A}{m_{A\cup j\setminus (l+1)}}\\
            &=0,
        \end{align*}
        where in the first equality we used \cref{lemma2.1} (1).    Similarly, the coefficient of $\epsilon_{A\cup j\setminus i}$ where $j\notin A \cup t$ and $i\in A\setminus [l+1]$ is
        \begin{align*}
            &\sign(j,A) a_{sj} \frac{m_j m_A}{m_{A\cup j\setminus i}} \sign(i,A\cup j\setminus i) + \sign(i,A\setminus i) \sign(j,A\setminus i) a_{sj} \frac{m_j m_A}{m_{A\cup j\setminus i}}\\
            &=(-1)^{|A|}\sign(A,j) a_{sj}  \frac{m_j m_A}{m_{A\cup j\setminus i}} \sign(i,A\cup j\setminus i) + (-1)^{|A|-1}\sign(A,j)\sign(i,A\cup j\setminus i) a_{sj} \frac{m_j m_A}{m_{A\cup j\setminus i}}\\
            &=0,
        \end{align*} where the first equality follows from \Cref{lemma2.1} (2) and (4).
    %\end{itemize}
    Combining the computations above, we obtain
    \[
    (\partial\sigma_{\mathbf{e}_s}+\sigma_{\mathbf{e}_s}\partial)(\epsilon_A) = a_s \epsilon_A + 0 = a_s\epsilon_A,
    \]
    as required.
\end{proof}

Following similar computations, we have the next result whose proof we provide in the appendix (see \Cref{app}).

\begin{lemma}\label{lem:Adoesnothavel+1del}
    For any $1\leq s\leq r$ and any $A \nsupseteq [l]$ such that $[l]\setminus A=\{t\}$ and $l+1\not \in A$, we have 
    \[(\partial\sigma_{\mathbf{e}_s}+\sigma_{\mathbf{e}_s}\partial)(\epsilon_A)=a_s \epsilon_A.
    \]
\end{lemma}

\begin{proposition}\label{prop:homotopy}
    For any $1\leq s\leq r$, we have
    \[
    \partial\sigma_{\mathbf{e}_s}+\sigma_{\mathbf{e}_s}\partial=a_s.
    \]
\end{proposition}
\begin{proof}
    Given $A\nsupseteq [l]$. If $|A\cap [l]| \leq l-2$, then the result follows from the proof of \cite[Theorem 3.2]{Sob23}, since the only formula for $(\sigma_{e_s})$ we will use is its first case, the same as that of the formula given in \cite{Sob23}. So now we can assume that $[l]\setminus A =\{t\}$ for some integer $t$. The result now follows from Lemmas \ref{lem:Ahasl+1del} and \ref{lem:Adoesnothavel+1del}.
\end{proof}

\begin{lemma}\label{lem:Ahasl+1sigma}
    For any $1\leq s<s'\leq r$ and any $A\nsupseteq [l]$ such that $|A\cap [l]|=l-2$ and $l+1\in A$, we have
    \[ (\sigma_{\mathbf{e}_s}\sigma_{\mathbf{e}_{s'}}+\sigma_{\mathbf{e}_{s'}}\sigma_{\mathbf{e}_s})(\epsilon_A)=0.
    \]
\end{lemma}

\begin{proof}
    Set $\{u,v\} = [l]\setminus A$. Then we have
    \begin{align*}
        &\ \ \ \ (\sigma_{\mathbf{e}_s}\sigma_{\mathbf{e}_{s'}}) (\epsilon_A) \\
        &= \sigma_{\mathbf{e}_s} \big( \sum_{j\notin A} \sign(j,A) a_{s'j}\frac{m_j m_A}{m_{A\cup j}}  \epsilon_{A\cup j}\big)\\
        &= \sigma_{\mathbf{e}_s} \big( \sign(u,A) a_{s'u}\frac{m_u m_A}{m_{A\cup u}}   \epsilon_{A\cup u}\big) +\sigma_{\mathbf{e}_s} \big( \sign(v,A) a_{s'v}m_v \frac{m_v m_A}{m_{A\cup v}}   \epsilon_{A\cup v}\big) + \sigma_{\mathbf{e}_s} \big( \sum_{j\notin A\cup u\cup v} \sign(j,A) a_{s'j}\frac{m_jm_A}{m_{A\cup j}}    \epsilon_{A\cup j}\big)\\
        &= \!\begin{multlined}[t][.5\displaywidth]
                \big(\sign(u,A)a_{s'u} \sum_{j\notin A\cup u \cup v} \sign(j,A\cup u) a_{sj}\frac{m_u m_j m_A}{m_{A\cup u\cup j}}   \epsilon_{A\cup u \cup j} \big) + \big( \sign(v,A)a_{s'v}\sum_{j\notin A\cup u \cup v} \sign(j,A\cup v) a_{sj}\frac{m_v m_j m_A}{m_{A\cup v\cup j}}   \epsilon_{A\cup v \cup j} \big)\\
                +\big( \sum_{j\notin A\cup u\cup v} \sign(j,A) a_{s'j} \sum_{i\notin A\cup j} \sign(i,A\cup j) a_{si} \frac{m_j m_i m_A}{m_{A\cup j\cup i}}   \epsilon_{A\cup j\cup i}  \big)
        \end{multlined}\\
        &= \!\begin{multlined}[t][.5\displaywidth]
                \big(\sum_{j\notin A\cup u \cup v} 
                \sign(u,A) \sign(j,A\cup u) a_{sj}a_{s'u}\frac{m_u m_j m_A}{m_{A\cup u\cup j}}  \epsilon_{A\cup u \cup j} \big) + \big( \sum_{j\notin A\cup u \cup v}\sign(v,A) \sign(j,A\cup v) a_{sj}a_{s'v}\frac{m_v m_j m_A}{m_{A\cup v\cup j}}  \epsilon_{A\cup v \cup j} \big)\\
                +\big( \sum_{j\notin A\cup u\cup v} \sign(j,A)  \sign(u,A\cup j) a_{s'j}a_{su}\frac{m_u m_j m_A}{m_{A\cup u\cup j}}   \epsilon_{A\cup u\cup j}  \big) + \big( \sum_{j\notin A\cup u\cup v} \sign(j,A)\sign(v,A\cup j)  a_{s'j}a_{sv}\frac{m_v m_j m_A}{m_{A\cup v\cup j}}  \epsilon_{A\cup v\cup j}  \big)\\
                +\big( \sum_{i,j\notin A\cup u\cup v, i\neq j} \sign(j,A)\sign(i,A\cup j) a_{s'j}a_{si}\frac{m_j m_i m_A}{m_{A\cup j\cup i}}    \epsilon_{A\cup j\cup i}  \big) +\big( \sum_{i,j\notin A\cup u\cup v, i\neq j} \sign(i,A)\sign(j,A\cup i) a_{s'i}a_{sj}\frac{m_j m_i m_A}{m_{A\cup j\cup i}}    \epsilon_{A\cup j\cup i}  \big).
        \end{multlined}
    \end{align*}
    We denote this expression by $F(s,s')$. There are 6 summands in it, and we will call them $F_i(s,s')$ for $i\in [6]$, in that order. The expression for $(\sigma_{\mathbf{e}_{s'}}\sigma_{\mathbf{e}_{s}})(\epsilon_A)$ is obtained by switching $s,s'$ in the previous calculation. Now, using \cref{lemma2.1} (4), we observe that
    \begin{align*}
        &F_1(s,s') + F_3(s',s) =0,
        &F_2(s,s') + F_4(s',s) =0,\quad \quad \quad \quad \quad
        &F_3(s,s') + F_1(s',s) =0,\\
        &F_4(s,s') + F_2(s',s) =0,
        &F_5(s,s') + F_6(s',s) =0,\quad \quad \quad \quad \quad
        &F_6(s,s') + F_5(s',s) =0.
    \end{align*}
    Thus $(\sigma_{\mathbf{e}_s}\sigma_{\mathbf{e}_{s'}}+\sigma_{\mathbf{e}_{s'}}\sigma_{\mathbf{e}_s})(\epsilon_A)= F(s,s')+F(s',s) =0$, as required.\qedhere    
    % Similarly, by switching $s$ and $s'$, we can see that 
    % \begin{align*}
    %     &(\sigma_{\mathbf{e}_{s'}}\sigma_{\mathbf{e}_{s}}) (\epsilon_A)\\
    %     &=\!\begin{multlined}[t][.5\displaywidth]
    %             \big(\sum_{j\notin A\cup u \cup v} 
    %             \sign(u,A) \sign(j,A\cup u) a_{s'j}a_{su}\frac{m_u m_j m_A}{m_{A\cup u\cup j}}  \epsilon_{A\cup u \cup j} \big) + \big( \sum_{j\notin A\cup u \cup v}\sign(v,A) \sign(j,A\cup v) a_{s'j}a_{sv}\frac{m_v m_j m_A}{m_{A\cup v\cup j}}  \epsilon_{A\cup v \cup j} \big)\\
    %             +\big( \sum_{j\notin A\cup u\cup v} \sign(j,A)  \sign(u,A\cup j) a_{sj}a_{s'u}\frac{m_u m_j m_A}{m_{A\cup u\cup j}}   \epsilon_{A\cup u\cup j}  \big) + \big( \sum_{j\notin A\cup u\cup v} \sign(j,A)\sign(v,A\cup j)  a_{sj}a_{s'v}\frac{m_v m_j m_A}{m_{A\cup v\cup j}}  \epsilon_{A\cup v\cup j}  \big)\\
    %             +\big( \sum_{i,j\notin A\cup u\cup v, i\neq j} \sign(j,A)\sign(i,A\cup j) a_{sj}a_{s'i}\frac{m_j m_i m_A}{m_{A\cup j\cup i}}    \epsilon_{A\cup j\cup i}  \big) +\big( \sum_{i,j\notin A\cup u\cup v, i\neq j} \sign(i,A)\sign(j,A\cup i) a_{si}a_{s'j}\frac{m_j m_i m_A}{m_{A\cup j\cup i}}    \epsilon_{A\cup j\cup i}  \big)
    %     \end{multlined}
    % \end{align*}
    % The result then follows immediately by combining the like terms and using \Cref{lemma2.1}. 
\end{proof}

Again, following similar computations, we get the next few results. We provide most of the proofs in the \hyperref[app]{Appendix}. 

\begin{lemma}\label{lem:Adoesnothavel+1sigma}
    For any $1\leq s<s'\leq r$ and any $A\nsupseteq [l]$ such that $|A\cap [l]|=l-2$ and $l+1\notin A$, we have
    \[(\sigma_{\mathbf{e}_s}\sigma_{\mathbf{e}_{s'}}+\sigma_{\mathbf{e}_{s'}}\sigma_{\mathbf{e}_s})(\epsilon_A)=0.
    \]
\end{lemma}

\begin{lemma}\label{lem:l-1Ahasl+1sigma}
    For any $1\leq s<s'\leq r$ and any $A\nsupseteq [l]$ such that $|A\cap [l]|=l-1$ and $l+1\in A$, we have
    \[
(\sigma_{\mathbf{e}_s}\sigma_{\mathbf{e}_{s'}}+\sigma_{\mathbf{e}_{s'}}\sigma_{\mathbf{e}_s})(\epsilon_A)=0.
    \]
\end{lemma}

% \begin{proof}
%     Set $\{t\}=[l]\setminus A$. Then we have
%     \begin{align*}
%         &\ \ \ \ (\sigma_{\mathbf{e}_s}\sigma_{\mathbf{e}_{s'}})(\epsilon_A)\\
%         &=\sigma_{\mathbf{e}_s}(\sum_{j\notin A\cup t} \sign(j,A) a_{s'j}m_j  \epsilon_{A\cup j})\\
%         &= \sum_{j\notin A\cup t} \sign(j,A) a_{s'j}m_j \sum_{i\notin A\cup t\cup j} \sign(i,A\cup j) a_{si}m_i  \epsilon_{A\cup i\cup j}  \\
%         &=\sum_{i,j\notin A\cup t, i\neq j} \sign(j,A)\sign(i,A\cup j) a_{s'j}a_{si}m_im_j     \epsilon_{A\cup i\cup j}
%     \end{align*}
%     Then we immediately obtain the formula for $ (\sigma_{\mathbf{e}_{s'}}\sigma_{\mathbf{e}_s})(\epsilon_A)$ by switching $s$ and $s'$. The result then follows immediately.
% \end{proof}

\begin{lemma}\label{lem:l-1Adoesnothavel+1sigma}
    For any $1\leq s<s'\leq r$ and any $A\nsupseteq [l]$ such that $|A\cap [l]|=l-1$ and $l+1\notin A$, we have
    \[
    (\sigma_{\mathbf{e}_s}\sigma_{\mathbf{e}_{s'}}+\sigma_{\mathbf{e}_{s'}}\sigma_{\mathbf{e}_s})(\epsilon_A)=0.
    \]
\end{lemma}

\begin{proposition}\label{prop:square-zero-1}
    For any $1\leq s<s'\leq r$, we have
    \[
    \sigma_{\mathbf{e}_s}\sigma_{\mathbf{e}_{s'}}+\sigma_{\mathbf{e}_{s'}}\sigma_{\mathbf{e}_s}=0.
    \]
\end{proposition}
\begin{proof}
    Given $A\nsupseteq [l]$. If $|A\cap [l]| \leq l-3$, then the result follows from the proof of \cite[Theorem 3.2]{Sob23}, since the only formula for $(\sigma_{e_s})$ we will use is its first case, the same as that of the formula given in \cite{Sob23}. The remaining two cases are when $|A\cap [l]|$ equals either $l-2$ or $l-1$. The result then follows from Lemmas \ref{lem:Ahasl+1sigma}, \ref{lem:Adoesnothavel+1sigma}, \ref{lem:l-1Ahasl+1sigma} and \ref{lem:l-1Adoesnothavel+1sigma}.
\end{proof}

\begin{proposition}\label{prop:square-zero-2}
    For any $1\leq s\leq r$, we have
    \[
    \sigma_{\mathbf{e}_s}^2=0.
    \]
\end{proposition}
\begin{proof}
    This follows by setting $s'=s$ in the formula for $\sigma_{\mathbf{e}_s}\sigma_{\mathbf{e}_{s'}}$ in the proofs of \Cref{lem:Ahasl+1sigma}, \ref{lem:Adoesnothavel+1sigma}, \ref{lem:l-1Ahasl+1sigma} and \ref{lem:l-1Adoesnothavel+1sigma}. Here we illustrate one case in detail. With the hypothesis for $A$ in \cref{lem:Ahasl+1sigma}, set $\{u,v\} = [l]\setminus A$. Then we have
    \begin{align*}
        &\ \ \ \ (\sigma_{\mathbf{e}_s}\sigma_{\mathbf{e}_{s'}}) (\epsilon_A) \\
        &= \!\begin{multlined}[t][.5\displaywidth]
                \big(\sum_{j\notin A\cup u \cup v} 
                \sign(u,A) \sign(j,A\cup u) a_{sj}a_{s'u}\frac{m_u m_j m_A}{m_{A\cup u\cup j}}  \epsilon_{A\cup u \cup j} \big) + \big( \sum_{j\notin A\cup u \cup v}\sign(v,A) \sign(j,A\cup v) a_{sj}a_{s'v}\frac{m_v m_j m_A}{m_{A\cup v\cup j}}  \epsilon_{A\cup v \cup j} \big)\\
                +\big( \sum_{j\notin A\cup u\cup v} \sign(j,A)  \sign(u,A\cup j) a_{s'j}a_{su}\frac{m_u m_j m_A}{m_{A\cup u\cup j}}   \epsilon_{A\cup u\cup j}  \big) + \big( \sum_{j\notin A\cup u\cup v} \sign(j,A)\sign(v,A\cup j)  a_{s'j}a_{sv}\frac{m_v m_j m_A}{m_{A\cup v\cup j}}  \epsilon_{A\cup v\cup j}  \big)\\
                +\big( \sum_{i,j\notin A\cup u\cup v, i\neq j} \sign(j,A)\sign(i,A\cup j) a_{s'j}a_{si}\frac{m_j m_i m_A}{m_{A\cup j\cup i}}    \epsilon_{A\cup j\cup i}  \big) +\big( \sum_{i,j\notin A\cup u\cup v, i\neq j} \sign(i,A)\sign(j,A\cup i) a_{s'i}a_{sj}\frac{m_j m_i m_A}{m_{A\cup j\cup i}}    \epsilon_{A\cup j\cup i}  \big)
        \end{multlined}\\
        &= F_1(s,s')+F_2(s,s')+F_3(s,s')+F_4(s,s')+F_5(s,s')+F_6(s,s')
    \end{align*}
        from the proof of \cref{lem:Ahasl+1sigma}, where the functions are set up to be equal to the 6 summands, respectively.  Using \cref{lemma2.1} (4), we observe that
    \begin{align*}
        F_1(s,s) + F_3(s,s) =0,\quad \quad 
        F_2(s,s) + F_4(s,s) =0,\quad \quad 
        F_5(s,s) + F_6(s,s) =0.
    \end{align*}
    Thus $(\sigma_{\mathbf{e}_s}\sigma_{\mathbf{e}_{s}}) (\epsilon_A) = F_1(s,s)+F_2(s,s)+F_3(s,s)+F_4(s,s)+F_5(s,s)+F_6(s,s)  =0$, as required.
\end{proof}

\noindent\textsc{Proof of \Cref{thm:system-pivot}:} We verify the three conditions of a system of higher homotopies described in \Cref{ES}. The condition (1) follows from the definition of $(\sigma)$. The condition (2) follows from \cref{prop:homotopy}. Finally, the condition (3) follows from Propositions \ref{prop:square-zero-1} and \ref{prop:square-zero-2}. \hfill $\square$

We end this section with new bounds on the Betti numbers of $R/I$ over $R$ using the new resolution we obtained by applying the Eisenbud-Shamash construction from \Cref{ES} to the pivot resolutions.

\begin{theorem}
    For any integer $i$, we have
    \[
    \beta_{2i}^R(R/I)\leq \sum_{j=0}^i \big( \binom{q}{2i} -\binom{q-\scarfnumber(I)}{2i-\scarfnumber(I)} \big) \binom{r+i-j-1}{r-1}
    \]
    and
    \[
    \beta_{2i+1}^R(R/I)\leq \sum_{j=0}^i \big( \binom{q}{2i+1} -\binom{q-\scarfnumber(I)}{2i+1-\scarfnumber(I)} \big) \binom{r+i-j-1}{r-1}.
    \]
\end{theorem}
\begin{proof}
    This follows immediately from the same rank counting procedure as in \cite[Proof of Corollary 3.3]{Sob23}, using the ranks of the pivot resolution from \cref{cor:Betti-minimal-pivot} instead of those of the Taylor resolution.
\end{proof}

\section*{Appendix}\label{app}

In this appendix, we provide the proofs of some of the lemmas from \Cref{system}. For convenience of notation, we symbolically write $\epsilon_A$ to denote $\frac{\epsilon_A}{m_A}$. Thus, the homotopy formula now looks like 
\begin{align*}
        \sigma_{e_{s}}(\epsilon_A)=\begin{cases}
            \sum_{j\notin A} \sign(j,A) a_{sj}m_j  \epsilon_{A\cup j} & \text{if } |A\cap [l]|\neq l-1\\
            \sum_{j\notin A\cup t} \sign(j,A) a_{sj}m_j  \epsilon_{A\cup j} & \text{if } [l] \setminus A = \{t\} \text{ and } l+1\in A \\
            \sum_{j\notin A\cup t} \sign(j,A) a_{sj}m_j  \epsilon_{A\cup j} +(-1)^{l+1} \sign(t,A)a_{st}m_t \sum_{i\in [l]} \sign(i,A\cup t\setminus i) \epsilon_{A\cup t\cup  (l+1)\setminus i} & \text{if } [l] \setminus A = \{t\} \text{ and } l+1\notin A
        \end{cases}
    \end{align*}
    and
\[
\sigma_{\mathbf{0}}=\delta, \qquad\sigma_{\mathbf{u}}=0 \text{ if } |\mathbf{u}|\geq 2.
\]
Also, the Taylor differentials (and also the pivot differentials, being restrictions of those of Taylor) become
\[
\partial(\epsilon_A) = \sum_{j\in A} \sign(j,A\setminus j) \epsilon_{A\setminus j}.
\]
% \begin{lemma}
%     For any $1\leq s\leq r$ and any $A\nsupseteq [l]$ such that $[l]\setminus A=\{t\}$ and $l+1\notin A$, we have
%     \[
%     (\partial\sigma_{\mathbf{e}_s}+\sigma_{\mathbf{e}_s}\partial)(\epsilon_A)=a_s \epsilon_A.
%     \]
% \end{lemma}
%\begin{proof}
\vspace{0.2cm}

 \noindent\textsc{Proof of \Cref{lem:Adoesnothavel+1del}:}  Using the formulae, we obtain 
    \begin{align*}&(\partial\sigma_{\mathbf{e}_s}+\sigma_{\mathbf{e}_s}\partial)(\epsilon_A)\\
        &= \partial\big(\sum_{j\notin A\cup t} \sign(j,A) a_{sj}m_j  \epsilon_{A\cup j} +(-1)^{l+1} \sign(t,A)a_{st}m_t \sum_{i\in [l]} \sign(i,A\cup t\setminus i) \epsilon_{A\cup t\cup  (l+1)\setminus i}\big) + \sigma_{\mathbf{e}_s} \big( \sum_{i\in A} \sign(i,A\setminus i) \epsilon_{A\setminus i} \big)\\
        &=\!\begin{multlined}[t][.5\displaywidth]
            \partial\big(\sum_{j\notin A\cup t} \sign(j,A) a_{sj}m_j  \epsilon_{A\cup j} \big) + \partial \big( (-1)^{l+1} \sign(t,A)a_{st}m_t \sum_{i\in [l]} \sign(i,A\cup t\setminus i) \epsilon_{A\cup t\cup  (l+1)\setminus i}\big)\\
            + \sigma_{\mathbf{e}_s} \big( \sum_{i\in A\cap [l]} \sign(i,A\setminus i) \epsilon_{A\setminus i} \big)  + \sigma_{\mathbf{e}_s} \big( \sum_{i\in A\setminus [l+1]} \sign(i,A\setminus i) \epsilon_{A\setminus i} \big).
        \end{multlined}
    \end{align*}
We will rewrite each of the four summands. We start with the first:
\begin{align*}
        &\ \ \ \ \partial\big(\sum_{j\notin A\cup t} \sign(j,A) a_{sj}m_j  \epsilon_{A\cup j} \big)\\
        &=\sum_{j\notin A\cup t} \sign(j,A) a_{sj}m_j  
         \sum_{i\in A\cup j}\sign(i,A\cup j\setminus i)\epsilon_{A\cup j\setminus i} \\
         &=\big(\sum_{j\notin A\cup t} \sum_{i\in A}\sign(j,A)\sign(i,A\cup j\setminus i) a_{sj}m_j  
         \epsilon_{A\cup j\setminus i} \big) + (\sum_{j\notin A\cup t} \sign(j,A)^2 a_{sj}m_j  \epsilon_{A} )\\
         &=\big(\sum_{j\notin A\cup t} \sum_{i\in A}\sign(j,A)\sign(i,A\cup j\setminus i) a_{sj}m_j  
         \epsilon_{A\cup j\setminus i} \big) + \big(\sum_{j\notin A\cup t} a_{sj}m_j  \epsilon_{A} \big).
    \end{align*}
Next we simplify the second summand:
     \begin{align*}
        &\ \ \ \ \partial \big( (-1)^{l+1} \sign(t,A)a_{st}m_t \sum_{i\in [l]} \sign(i,A\cup t\setminus i) \epsilon_{A\cup t\cup  (l+1)\setminus i}\big)\\
        &=(-1)^{l+1} \sign(t,A)a_{st}m_t \big( \sum_{i\in [l]} \sign(i,A\cup t\setminus i) \sum_{k\in A\cup t\cup  (l+1)\setminus i} \sign(k, A\cup t\cup  (l+1)\setminus i\setminus k)\epsilon_{A\cup t\cup  (l+1)\setminus i\setminus k}\big) \\
        &=(-1)^{l+1} \sign(t,A)\big( \sum_{i\in [l]} \sum_{k\in A\cup t\cup  (l+1)\setminus i} \sign(i,A\cup t\setminus i)  \sign(k, A\cup t\cup  (l+1)\setminus i\setminus k)a_{st}m_t \epsilon_{A\cup t\cup  (l+1)\setminus i\setminus k}\big) \\
        &=\!\begin{multlined}[t][.5\displaywidth]
         (-1)^{l+1} \sign(t,A)\big( \sum_{i\in [l]} \sum_{k\in A\setminus i} \sign(i,A\cup t\setminus i)  \sign(k, A\cup t\cup  (l+1)\setminus i\setminus k)a_{st}m_t \epsilon_{A\cup t\cup  (l+1)\setminus i\setminus k}\big)  \\
         + (-1)^{l+1} \sign(t,A)\big( \sum_{i\in [l]\setminus t} \sign(i,A\cup t\setminus i)  \sign(t, A\cup  (l+1)\setminus i)a_{st}m_t \epsilon_{A\cup  (l+1)\setminus i}\big) \\
         +(-1)^{l+1} \sign(t,A)\big( \sum_{i\in [l]}\sign(i,A\cup t\setminus i)  \sign(l+1, A\cup t\setminus i)a_{st}m_t \epsilon_{A\cup t\setminus i}\big) 
        \end{multlined}\\
        &=\!\begin{multlined}[t][.5\displaywidth]
         (-1)^{l+1} \sign(t,A)\big( \sum_{i\in [l]\setminus t} \sum_{k\in A\setminus i} \sign(i,A\cup t\setminus i)  \sign(k, A\cup t\cup  (l+1)\setminus i\setminus k)a_{st}m_t \epsilon_{A\cup t\cup  (l+1)\setminus i\setminus k}\big)  \\
         +(-1)^{l+1} \sign(t,A)\big(  \sum_{k\in A} \sign(t,A)  \sign(k, A\cup  (l+1)\setminus k)a_{st}m_t \epsilon_{A\cup  (l+1)\setminus k}\big) \\
         + (-1)^{l+1} \sign(t,A)\big( \sum_{i\in [l]\setminus t} \sign(i,A\cup t\setminus i)  \sign(t, A\setminus i)a_{st}m_t \epsilon_{A\cup  (l+1)\setminus i}\big) \\
         +(-1)^{l+1} \sign(t,A)\big( \sum_{i\in [l]}\sign(i,A\cup t\setminus i)  (-1)^{l-1}a_{st}m_t \epsilon_{A\cup t\setminus i}\big) 
        \end{multlined}\\
         &=\!\begin{multlined}[t][.5\displaywidth]
         (-1)^{l+1} \sign(t,A)\big( \sum_{i\in [l]\setminus t} \sum_{k\in A\setminus i} \sign(i,A\cup t\setminus i)  \sign(k, A\cup t\cup  (l+1)\setminus i\setminus k)a_{st}m_t \epsilon_{A\cup t\cup  (l+1)\setminus i\setminus k}\big)  \\
         +(-1)^{l+1} \big(  \sum_{i\in A}  \sign(i, A\cup  (l+1)\setminus i)a_{st}m_t \epsilon_{A\cup  (l+1)\setminus i}\big) \\
         + (-1)^{l+1} \sign(t,A)\big( \sum_{i\in [l]\setminus t} \sign(i,A\cup t\setminus i)  \sign(t, A\setminus i)a_{st}m_t \epsilon_{A\cup  (l+1)\setminus i}\big) \\
         + \sign(t,A)\big( \sum_{i\in [l]\setminus t}\sign(i,A\cup t\setminus i)  a_{st}m_t \epsilon_{A\cup t\setminus i}\big) + \big(  a_{st}m_t \epsilon_{A}\big) 
         \end{multlined}.
    \end{align*}
We remark that  the basis element $\epsilon_{A\cup t\cup (l+1)\setminus i\setminus k}$ appears two times for each pair $i,k\in [l]\setminus t$. We now show that the sum of the two coefficients is actually zero:
\begin{align*}
 &(-1)^{l+1} \sign(t,A)  \big( \sign(i,A\cup t\setminus i)  \sign(k, A\cup t\cup  (l+1)\setminus i\setminus k)+ \sign(k,A\cup t\setminus k)  \sign(i, A\cup t\cup  (l+1)\setminus i\setminus k) \big) a_{st}m_t  \\
        &=(-1)^{l+1} \sign(t,A)  \big( \sign(i,A\cup t\setminus i)  \sign(k, A\cup t\setminus i\setminus k)+ \sign(k,A\cup t\setminus k)  \sign(i, A\cup t\setminus i\setminus k) \big) a_{st}m_t\\
        &=0,
\end{align*}
which follows from the first equality of \cref{lemma2.1} (4).
Moreover, the basis element $\epsilon_{A\cup(l+1)\setminus i}$ also appears twice for each $i\in [l]\setminus t$. We will show that the sum of these two coefficients is also zero:
    \begin{align*}
        &(-1)^{l+1} \big(   \sign(i, A\cup  (l+1)\setminus i)a_{st}m_t \big) + (-1)^{l+1} \sign(t,A)\big( \sign(i,A\cup t\setminus i)  \sign(t, A\setminus i)a_{st}m_t \big) \\
        &= (-1)^{l+1} \big(   \sign(i, A\setminus i)a_{st}m_t \big) + (-1)^{l+1} \sign(t,A)\big( -\sign(t,A)  \sign(i, A\setminus i)a_{st}m_t \big)\\
        &= (-1)^{l+1} \big(   \sign(i, A\setminus i)a_{st}m_t \big) - (-1)^{l+1} \big(  \sign(i, A\setminus i)a_{st}m_t \big)\\
        &=0,
    \end{align*}
    where the first equality comes from the first equality of \cref{lemma2.1} (4). Thus the second  summand becomes
    \begin{align*}
        &\partial \big( (-1)^{l+1} \sign(t,A)a_{st}m_t \sum_{i\in [l]} \sign(i,A\cup t\setminus i) \epsilon_{A\cup t\cup  (l+1)\setminus i}\big)\\
         &=\!\begin{multlined}[t][.5\displaywidth]
         (-1)^{l+1} \sign(t,A)\big( \sum_{i\in [l]\setminus t} \sum_{k\in A\setminus [l+1]} \sign(i,A\cup t\setminus i)  \sign(k, A\cup t\cup  (l+1)\setminus i\setminus k)a_{st}m_t \epsilon_{A\cup t\cup  (l+1)\setminus i\setminus k}\big)  \\
         +(-1)^{l+1} \big(  \sum_{i\in A\setminus [l+1]}  \sign(i, A\cup  (l+1)\setminus i)a_{st}m_t \epsilon_{A\cup  (l+1)\setminus i}\big) 
         + \sign(t,A)\big( \sum_{i\in [l]\setminus t}\sign(i,A\cup t\setminus i)  a_{st}m_t \epsilon_{A\cup t\setminus i}\big) + \big(  a_{st}m_t \epsilon_{A}\big) 
         \end{multlined}.
    \end{align*}
Next we work on the third summand:
\begin{align*}
        & \sigma_{\mathbf{e}_s} \big( \sum_{i\in A\cap [l]} \sign(i,A\setminus i) \epsilon_{A\setminus i} \big) \\
        &=\big( \sum_{i\in A\cap [l]} \sign(i,A\setminus i) \sum_{j\notin A\setminus i} \sign(j,A\setminus i) a_{sj}m_j\epsilon_{A\cup j\setminus i} \big)\\
        &=\big( \sum_{i\in [l]\setminus t}\sum_{j\notin A\setminus i} \sign(i,A\setminus i)  \sign(j,A\setminus i) a_{sj}m_j\epsilon_{A\cup j\setminus i} \big)\\
        &=\big( \sum_{i\in [l]\setminus t}\sum_{j\notin A}  \sign(i,A\setminus i) \sign(j,A\setminus i) a_{sj}m_j\epsilon_{A\cup j\setminus i} \big) + \big( \sum_{i\in [l]\setminus t} \sign(i,A\setminus i) \sign(i,A\setminus i) a_{si}m_i\epsilon_{A} \big)\\
        &=\big( \sum_{j\notin A\cup t} \sum_{i\in [l]\setminus t}  \sign(i,A\setminus i) \sign(j,A\setminus i) a_{sj}m_j\epsilon_{A\cup j\setminus i} \big) + \big( \sum_{i\in [l]\setminus t}  \sign(i,A\setminus i) \sign(t,A\setminus i) a_{st}m_t\epsilon_{A\cup t\setminus i} \big) + \big( \sum_{i\in [l]\setminus t} a_{si}m_i\epsilon_{A} \big).
    \end{align*}
    Finally we work on the fourth summand:
     \begin{align*}
        &\sigma_{\mathbf{e}_s} \big( \sum_{i\in A\setminus [l+1]} \sign(i,A\setminus i) \epsilon_{A\setminus i} \big)\\
        &=\!\begin{multlined}[t][.5\displaywidth]
            \sum_{i\in A\setminus [l+1]} \sign(i,A\setminus i)  \big( \sum_{j\notin A\cup t\setminus i} \sign(j,A\setminus i) a_{sj}m_j  \epsilon_{A\cup j\setminus i} +(-1)^{l+1} \sign(t,A\setminus i)a_{st}m_t \sum_{k\in [l]} \sign(k,A\cup t\setminus i\setminus k) \epsilon_{A\cup t\cup  (l+1)\setminus i\setminus k} \big)
        \end{multlined}\\
        &=\!\begin{multlined}[t][.5\displaywidth]
            \sum_{i\in A\setminus [l+1]} \sign(i,A\setminus i)  \big( \sum_{j\notin A\cup t\setminus i} \sign(j,A\setminus i) a_{sj}m_j  \epsilon_{A\cup j\setminus i} \big)\\
            +(-1)^{l+1}\sum_{i\in A\setminus [l+1]} \sign(i,A\setminus i)  \big( \sign(t,A\setminus i)a_{st}m_t \sum_{k\in [l]} \sign(k,A\cup t\setminus i\setminus k) \epsilon_{A\cup t\cup  (l+1)\setminus i\setminus k} \big)
        \end{multlined}\\
        &=\!\begin{multlined}[t][.5\displaywidth]
            \big(\sum_{i\in A\setminus [l+1]} \sum_{j\notin A\cup t\setminus i} \sign(i,A\setminus i)    \sign(j,A\setminus i) a_{sj}m_j  \epsilon_{A\cup j\setminus i} \big)\\
            +(-1)^{l+1}  \big( \sum_{i\in A\setminus [l+1]} \sum_{k\in [l]} \sign(i,A\setminus i)\sign(t,A\setminus i)\sign(k,A\cup t\setminus i\setminus k)a_{st}m_t  \epsilon_{A\cup t\cup  (l+1)\setminus i\setminus k} \big).
        \end{multlined}
    \end{align*} 
    Here in the first part, we separate the case $j=i$ from it, and in the second part, we separate the case $k=t$:
    \begin{align*}
        &\ \ \ \ \!\begin{multlined}[t][.5\displaywidth]
            \big(\sum_{i\in A\setminus [l+1]} \sum_{j\notin A\cup t} \sign(i,A\setminus i)    \sign(j,A\setminus i) a_{sj}m_j  \epsilon_{A\cup j\setminus i} \big) + \big(\sum_{i\in A\setminus [l+1]} \sign(i,A\setminus i)    \sign(i,A\setminus i) a_{si}m_i  \epsilon_{A} \big)\\
            +(-1)^{l+1}  \big( \sum_{i\in A\setminus [l+1]} \sum_{k\in [l]\setminus t} \sign(i,A\setminus i)\sign(t,A\setminus i)\sign(k,A\cup t\setminus i\setminus k)a_{st}m_t  \epsilon_{A\cup t\cup  (l+1)\setminus i\setminus k} \big)\\
            + (-1)^{l+1}  \big( \sum_{i\in A\setminus [l+1]} \sign(i,A\setminus i)\sign(t,A\setminus i)\sign(t,A\setminus i)a_{st}m_t  \epsilon_{A\cup  (l+1)\setminus i} \big)
        \end{multlined}\\
        &=\!\begin{multlined}[t][.5\displaywidth]
            \big(\sum_{j\notin A\cup t} \sum_{i\in A\setminus [l+1]} \sign(i,A\setminus i)    \sign(j,A\setminus i) a_{sj}m_j  \epsilon_{A\cup j\setminus i} \big) + \big(\sum_{i\in A\setminus [l+1]} a_{si}m_i  \epsilon_{A} \big)\\
            +(-1)^{l+1} \sum_{i\in [l]\setminus t} \big( \sum_{k\in A\setminus [l+1]}  \sign(k,A\setminus k)\sign(t,A\setminus k)\sign(i,A\cup t\setminus i\setminus k)a_{st}m_t  \epsilon_{A\cup t\cup  (l+1)\setminus i\setminus k} \big)\\
            + (-1)^{l+1}  \big( \sum_{i\in A\setminus [l+1]} \sign(i,A\setminus i)a_{st}m_t  \epsilon_{A\cup  (l+1)\setminus i} \big),
        \end{multlined}
    \end{align*} 
    where the last equality comes from applying \cref{lemma2.1} (4) twice. In the rewritten forms we can see the coefficients clearly. We will simplify the coefficient for each basis element.
    \begin{itemize}
        \item The basis element $\epsilon_A$ appears in all four summands. We combine all of its coefficients:
        \begin{align*}
            \sum_{j\notin A\cup t}a_{sj}m_j +  a_{st}m_t 
                + \sum_{i\in [l]\setminus t} a_{si} m_i
                + \sum_{i\in A\setminus [l+1]} a_{si}m_i &=\sum_{j\notin A, j\geq l+1}a_{sj}m_j + a_{st}m_t   + \sum_{i\in [l]\setminus t} a_{si} m_i
                + \sum_{i\in A, i\geq l+1} a_{si}m_i \\
                &= \sum_{i\in [q]} a_{si}m_i = a_s.
        \end{align*}
        \item For each $j\notin A\cup t$ and each $i\in A$, the basis element $\epsilon_{A\cup j\setminus i}$ appears in the first summand and, depending on $i$, in the third or fourth summand. The coefficients, however, follow the same rule no matter what $i$ is. We combine all of its coefficients:
        \begin{align*}
            \big( \sign(j,A)\sign(i,A\cup j\setminus i) a_{sj}m_j \big) + \big( \sign(i,A\setminus i)\sign(j,A\setminus i) a_{sj}m_j\big) = 0.
        \end{align*}
        This follows from \cref{lemma2.1} (4).
        \item For each $i\in [l]\setminus t$ and each $k\in A\setminus [l+1]$, the basis element $\epsilon_{A\cup t\cup (l+1)\setminus i\setminus k}$ appears in the second and fourth summands. We combine all of its coefficients:
        \begin{align*}
            & \ \ \ \ \!\begin{multlined}[t][.5\displaywidth]
                \big( (-1)^{l+1}\sign(t,A) \sign(i,A\cup t\setminus i) \sign(k,A\cup t \cup (l+1)\setminus i \setminus k) a_{st}m_t \big)\\
                + \big( (-1)^{l+1} \sign(k,A\setminus k)\sign(t,A\setminus k)\sign(i,A\cup t\setminus i\setminus k) a_{st}m_t \big)
            \end{multlined}\\
            &=(-1)^{l+1}a_{st}m_t \big(  \sign(t,A) \sign(i,A\cup t\setminus i) \sign(k,A\cup t \cup (l+1)\setminus i \setminus k) + \sign(k,A\setminus k)\sign(t,A\setminus k)\sign(i,A\cup t\setminus i\setminus k) \big)\\
            &=(-1)^{l+1}a_{st}m_t \big(  \sign(t,A) \sign(i,A\cup t\setminus i) (-\sign(k,A\cup t\setminus i \setminus k)) + \sign(k,A\setminus k)\sign(t,A\setminus k)\sign(i,A\cup t\setminus i\setminus k) \big)\\
            &=(-1)^{l+1}a_{st}m_t \big(  -\sign(t,A) \sign(i,A\cup t\setminus i) \sign(k,A\cup t\setminus i \setminus k) + \sign(k,A\setminus k)\sign(t,A\setminus k)\sign(i,A\cup t\setminus i\setminus k) \big)\\
            &=(-1)^{l+1}a_{st}m_t \big(  -\sign(t,A) (\sign(i,A\cup t\setminus i) \sign(k,A\cup t\setminus i \setminus k)) + (\sign(k,A\setminus k)\sign(t,A\setminus k))\sign(i,A\cup t\setminus i\setminus k) \big)\\
            &=(-1)^{l+1}a_{st}m_t \big(  -\sign(t,A) (-\sign(k,A\cup t\setminus k) \sign(i,A\cup t\setminus i \setminus k)) + (-\sign(k,A\cup t\setminus k)\sign(t,A))\sign(i,A\cup t\setminus i\setminus k) \big)\\
            &=0,
        \end{align*}
        where we used \cref{lemma2.1} (1) and (4) for the equalities.
        \item For each $i\in A\setminus [l+1]$, the basis element $\epsilon_{A\cup (l+1)\setminus i}$ appears in the second and fourth summands. We combine all of its coefficients:
        \begin{align*}
            & \ \ \ \ \big( (-1)^{l+1} \sign(i,A\cup (l+1)\setminus i) a_{st}m_t \big) + \big( (-1)^{l+1}\sign(i,A\setminus i) a_{st}m_t \big)\\
            &=  \big( (-1)^{l+1} (-\sign(i,A\setminus i)) a_{st}m_t \big) + \big( (-1)^{l+1}\sign(i,A\setminus i) a_{st}m_t \big)=0,
        \end{align*}
        where we used \cref{lemma2.1} (1) for the first equality.
        \item For each $i\in [l]\setminus t$, the basis element $\epsilon_{A\cup t\setminus i}$ appears in the second and third summands. We combine all of its coefficients:
        \begin{align*}
             \big( \sign(t,A)\sign(i,A\cup t\setminus i)a_{st}m_t \big) + \big( \sign(i,A\setminus i)\sign(t,A\setminus i) a_{st}m_t \big)=0.
        \end{align*}
    \end{itemize} \hfill $\square$
%\end{proof}

% \begin{lemma}
%     For any $1\leq s<s'\leq r$ and any $A\nsupseteq [l]$ such that $|A\cap [l]|=l-2$ and $l+1\notin A$, we have
%     \[
%     (\sigma_{\mathbf{e}_s}\sigma_{\mathbf{e}_{s'}}+\sigma_{\mathbf{e}_{s'}}\sigma_{\mathbf{e}_s})(\epsilon_A)=0.
%     \]
% \end{lemma}

\vspace{0.2cm}
\noindent\textsc{Proof of \Cref{lem:Adoesnothavel+1sigma}:}
%\begin{proof}
    Set $\{u,v\} = [l]\setminus A$. Then we have
    \begin{align*}
        &\ \ \ \ (\sigma_{\mathbf{e}_s}\sigma_{\mathbf{e}_{s'}}) (\epsilon_A) \\
        &= \sigma_{\mathbf{e}_s} \big( \sum_{j\notin A} \sign(j,A) a_{s'j}m_j  \epsilon_{A\cup j}\big)\\
        &= \sigma_{\mathbf{e}_s} \big( \sign(u,A) a_{s'u}m_u  \epsilon_{A\cup u}\big) +\sigma_{\mathbf{e}_s} \big( \sign(v,A) a_{s'v}m_v  \epsilon_{A\cup v}\big) + \sigma_{\mathbf{e}_s} \big( \sum_{j\notin A\cup u\cup v} \sign(j,A) a_{s'j}m_j  \epsilon_{A\cup j}\big)\\
        &= \!\begin{multlined}[t][.5\displaywidth]
                \sign(u,A)a_{s'u}m_u \big(\sum_{j\notin A\cup u \cup v} \sign(j,A\cup u) a_{sj}m_j \epsilon_{A\cup u \cup j}  + (-1)^{l+1} \sign(v,A\cup u)a_{sv}m_v \sum_{i\in [l]} \sign(i,A\cup u\cup v\setminus i) \epsilon_{A\cup u\cup v\cup  (l+1)\setminus i} \big) \\
                + \sign(v,A)a_{s'v}m_v \big( \sum_{j\notin A\cup u \cup v} \sign(j,A\cup v) a_{sj}m_j \epsilon_{A\cup v \cup j} + (-1)^{l+1} \sign(u,A\cup v)a_{su}m_u \sum_{i\in [l]} \sign(i,A\cup u\cup v\setminus i) \epsilon_{A\cup u\cup v\cup  (l+1)\setminus i}\big)\\
                +\big( \sum_{j\notin A\cup u\cup v} \sign(j,A) a_{s'j}m_j \sum_{i\notin A\cup j} \sign(i,A\cup j) a_{si}m_i \epsilon_{A\cup j\cup i}  \big)
        \end{multlined}\\
        &=G_1(s,s')+G_2(s,s'),
    \end{align*}
    where $G_1(s,s')$ is the part that does not involve $(-1)^{l+1}$ and $G_2(s,s')$ is the part that does. We observe that $G_1(s,s')$ already appeared in the proof of \cref{lem:Ahasl+1sigma}, and from there we have $G_1(s,s')+G_1(s',s)=0$. We inspect the $G_2(s,s')$ part. We have
    \begin{align*}
        G_2(s,s')  &= \!\begin{multlined}[t][.5\displaywidth]
                (-1)^{l+1}\sign(u,A)\sign(v,A\cup u)a_{s'u}a_{sv}m_um_v \big(   \sum_{i\in [l]} \sign(i,A\cup u\cup v\setminus i) \epsilon_{A\cup u\cup v\cup  (l+1)\setminus i} \big)\\ + (-1)^{l+1} \sign(v,A)\sign(u,A\cup v)a_{s'v}a_{su}m_um_v \big( \sum_{i\in [l]} \sign(i,A\cup u\cup v\setminus i) \epsilon_{A\cup u\cup v\cup  (l+1)\setminus i}\big).
        \end{multlined}
    \end{align*}
    It now follows from \cref{lemma2.1} (4) that $G_2(s,s')+G_2(s',s)=0$. Thus 
    \[
(\sigma_{\mathbf{e}_s}\sigma_{\mathbf{e}_{s'}}+\sigma_{\mathbf{e}_{s'}}\sigma_{\mathbf{e}_s})(\epsilon_A) =\big( G_1(s,s')+G_2(s,s')\big)+ \big( G_1(s',s)+G_2(s',s) \big)=0,
    \]
    as required. \hfill $\square$
%\end{proof}

% \begin{lemma}
%     For any $1\leq s<s'\leq r$ and any $A\nsupseteq [l]$ such that $|A\cap [l]|=l-1$ and $l+1\in A$, we have
%     \[
%     (\sigma_{\mathbf{e}_s}\sigma_{\mathbf{e}_{s'}}+\sigma_{\mathbf{e}_{s'}}\sigma_{\mathbf{e}_s})(\epsilon_A)=0.
%     \]
% \end{lemma}

\vspace{0.2cm}
\noindent \textsc{Proof of \Cref{lem:l-1Ahasl+1sigma}:}
%\begin{proof}
    Set $\{t\}=[l]\setminus A$. Then we have
    \begin{align*}                      
        (\sigma_{\mathbf{e}_s}\sigma_{\mathbf{e}_{s'}})(\epsilon_A)&=\sigma_{\mathbf{e}_s}(\sum_{j\notin A\cup t} \sign(j,A) a_{s'j}m_j  \epsilon_{A\cup j})\\
        &= \sum_{j\notin A\cup t} \sign(j,A) a_{s'j}m_j \sum_{i\notin A\cup t\cup j} \sign(i,A\cup j) a_{si}m_i  \epsilon_{A\cup i\cup j}  \\
        &=\sum_{i,j\notin A\cup t, i\neq j} \sign(j,A)\sign(i,A\cup j) a_{s'j}a_{si}m_im_j     \epsilon_{A\cup i\cup j}.
    \end{align*}
    Then we immediately obtain the formula for $ (\sigma_{\mathbf{e}_{s'}}\sigma_{\mathbf{e}_s})(\epsilon_A)$ by switching $s$ and $s'$. The result then follows immediately from \cref{lemma2.1} (4).
    \hfill $\square$ 
%\end{proof}

% \begin{lemma}
%     For any $1\leq s<s'\leq r$ and any $A\nsupseteq [l]$ such that $|A\cap [l]|=l-1$ and $l+1\notin A$, we have
%     \[
%     (\sigma_{\mathbf{e}_s}\sigma_{\mathbf{e}_{s'}}+\sigma_{\mathbf{e}_{s'}}\sigma_{\mathbf{e}_s})(\epsilon_A)=0.
%     \]
% \end{lemma}

\vspace{0.2cm}
\noindent\textsc{Proof of \Cref{lem:l-1Adoesnothavel+1sigma}:}
%\begin{proof}
    Set $\{t\}=[l]\setminus A$. Then we have
    \begin{align*}
        &\ \ \ \ (\sigma_{\mathbf{e}_s}\sigma_{\mathbf{e}_{s'}})(\epsilon_A)\\
        &=\sigma_{\mathbf{e}_s}\big( \sum_{j\notin A\cup t} \sign(j,A) a_{s'j}m_j  \epsilon_{A\cup j} +(-1)^{l+1} \sign(t,A)a_{s't}m_t \sum_{i\in [l]} \sign(i,A\cup t\setminus i) \epsilon_{A\cup t\cup  (l+1)\setminus i} \big)\\
        &=\sigma_{\mathbf{e}_s}\big( \sum_{j\notin A\cup t} \sign(j,A) a_{s'j}m_j  \epsilon_{A\cup j}\big) +(-1)^{l+1} \sign(t,A)a_{s't}m_t\sigma_{\mathbf{e}_s}\big(  \sum_{i\in [l]} \sign(i,A\cup t\setminus i) \epsilon_{A\cup t\cup  (l+1)\setminus i} \big)\\
        &= \!\begin{multlined}[t][.5\displaywidth]
            \sigma_{\mathbf{e}_s}\big( \sum_{j\notin A\cup [l+1]} \sign(j,A) a_{s'j}m_j  \epsilon_{A\cup j}\big) + \sigma_{\mathbf{e}_s}\big( \sign(l+1,A) a_{s',l+1}m_{l+1}  \epsilon_{A\cup (l+1)}\big) \\
            +(-1)^{l+1} \sign(t,A)a_{s't}m_t\sigma_{\mathbf{e}_s}\big(  \sum_{i\in [l]} \sign(i,A\cup t\setminus i) \epsilon_{A\cup t\cup  (l+1)\setminus i} \big)
        \end{multlined}\\
        &= \!\begin{multlined}[t][.5\displaywidth]
            \sigma_{\mathbf{e}_s}\big( \sum_{j\notin A\cup [l+1]} \sign(j,A) a_{s'j}m_j  \epsilon_{A\cup j}\big) + (-1)^{l+1}\sigma_{\mathbf{e}_s}\big( a_{s',l+1}m_{l+1}  \epsilon_{A\cup (l+1)}\big) \\
            +(-1)^{l+1} \sign(t,A)a_{s't}m_t\sigma_{\mathbf{e}_s}\big(  \sum_{i\in [l]} \sign(i,A\cup t\setminus i) \epsilon_{A\cup t\cup  (l+1)\setminus i} \big)
        \end{multlined}.
\end{align*}
We now apply the formula for $\sigma_{\mathbf{e}_s}$ appropriately:
\begin{align*}
    \\
        &\ \ \ \  \!\begin{multlined}[t][.5\displaywidth]
             \sum_{j\notin A\cup [l+1]} \sign(j,A) a_{s'j}m_j \big( \sum_{k\notin A\cup t \cup j} \sign(k,A\cup j) a_{sk}m_k  \epsilon_{A\cup k\cup j} +(-1)^{l+1} \sign(t,A)a_{st}m_t \sum_{i\in [l]} \sign(i,A\cup t\cup j\setminus i) \epsilon_{A\cup t \cup j\cup  (l+1)\setminus i}\big) \\
            + (-1)^{l+1}a_{s', l+1}m_{l+1}\big(   \sum_{j\notin A\cup t \cup (l+1)} \sign(j,A\cup (l+1)) a_{sj}m_j  \epsilon_{A\cup j \cup (l+1)} \big) \\
            +(-1)^{l+1} \sign(t,A)a_{s't}m_t\sum_{i\in [l]} \sign(i,A\cup t\setminus i) \big(  \sum_{j\notin A\cup t\cup  (l+1)} \sign(j,A\cup t\cup  (l+1)\setminus i) a_{sj}m_j  \epsilon_{A\cup t\cup  (l+1) \cup j \setminus i} \big)
        \end{multlined}\\
        &=\!\begin{multlined}[t][.5\displaywidth]
             \big( \sum_{j\notin A\cup [l+1]} \sum_{k\notin A\cup t \cup j} \sign(j,A)  \sign(k,A\cup j) a_{s'j}a_{sk}m_jm_k  \epsilon_{A\cup k\cup j} \big) \\
             +(-1)^{l+1} \sign(t,A)   \big( \sum_{i\in [l]}\sum_{j\notin A\cup [l+1]} \sign(j,A) \sign(i,A\cup t\cup j\setminus i) a_{s'j}a_{st}m_jm_t \epsilon_{A\cup t \cup j\cup  (l+1)\setminus i}\big)\\
            + (-1)^{l+1}\big( \sum_{j\notin A\cup [l+1]} \sign(j,A\cup (l+1))    a_{s', l+1}  a_{sj}m_j m_{l+1} \epsilon_{A\cup j \cup (l+1)} \big) \\
            +(-1)^{l+1} \sign(t,A)\big(\sum_{i\in [l]}    \sum_{j\notin A\cup  [l+1]} \sign(i,A\cup t\setminus i)\sign(j,A\cup t\cup  (l+1)\setminus i) a_{s't}a_{sj}m_jm_t  \epsilon_{A\cup t\cup  (l+1) \cup j \setminus i} \big).
        \end{multlined}
\end{align*}
We separate the case $k=l+1$ from the first summand, and using \cref{lemma2.1} on some of the sign functions:
\begin{align*}
         &\!\begin{multlined}[t][.5\displaywidth]
             \big( \sum_{j,k\notin A\cup [l+1], j\neq k} \sign(j,A)  \sign(k,A\cup j) a_{s'j}a_{sk}m_jm_k  \epsilon_{A\cup k\cup j} \big) + (-1)^{l+1}\big( \sum_{j\notin A\cup [l+1]} \sign(j,A)  a_{s'j}a_{s, l+1}m_jm_{l+1}  \epsilon_{A\cup j\cup (l+1)} \big) \\
             +(-1)^{l+1} \sign(t,A)   \big( \sum_{i\in [l]}\sum_{j\notin A\cup [l+1]} \sign(j,A) \sign(i,A\cup t\setminus i) a_{s'j}a_{st}m_jm_t \epsilon_{A\cup t \cup j\cup  (l+1)\setminus i}\big)\\
            + (-1)^{l}\big( \sum_{j\notin A\cup [l+1]} \sign(j,A)    a_{s', l+1}  a_{sj}m_j m_{l+1} \epsilon_{A\cup j \cup (l+1)} \big) \\
            +(-1)^{l} \sign(t,A)\big(\sum_{i\in [l]}    \sum_{j\notin A\cup  [l+1]} \sign(i,A\cup t\setminus i)\sign(j,A) a_{s't}a_{sj}m_jm_t  \epsilon_{A\cup t\cup  (l+1) \cup j \setminus i} \big)
        \end{multlined}\\
            &=H_1(s,s')+H_2(s,s')+H_3(s,s')+H_4(s,s')+H_5(s,s'),
    \end{align*}
    where the new functions are set to be equal to the summands correspondingly. It then follows from \cref{lemma2.1} (4) that
    \begin{align*}
        H_1(s,s')+H_1(s',s)=0,\quad
        H_2(s,s')+H_4(s',s)=0,\quad
        H_3(s,s')+H_5(s',s)=0,\\
        H_4(s,s')+H_2(s',s)=0,\quad
        H_5(s,s')+H_3(s',s)=0. \quad \quad \quad \quad \quad \quad
    \end{align*}
    Thus the result follows. \hfill $\square$
%\end{proof}

\bibliographystyle{amsplain}
\bibliography{refs}

\providecommand{\bysame}{\leavevmode\hbox to3em{\hrulefill}\thinspace}
\providecommand{\MR}{\relax\ifhmode\unskip\space\fi MR }
% \MRhref is called by the amsart/book/proc definition of \MR.
\providecommand{\MRhref}[2]{%
  \href{http://www.ams.org/mathscinet-getitem?mr=#1}{#2}
}
\providecommand{\href}[2]{#2}
\begin{thebibliography}{10}

\bibitem{Avr81}
Luchezar~L. Avramov, \emph{Obstructions to the existence of multiplicative structures on minimal free resolutions}, American Journal of Mathematics \textbf{103} (1981), 1.

\bibitem{Avr98}
Luchezar~L. Avramov, \emph{Infinite free resolutions}, pp.~1--118, Birkh{\"a}user Basel, Basel, 1998.

\bibitem{Bat02}
Ekkehard Batzies, \emph{Discrete morse theory for cellular resolutions}, Ph.D. Thesis, University of Marburg (2002).

\bibitem{BW02}
Ekkehard Batzies and Volkmar Welker, \emph{Discrete {M}orse theory for cellular resolutions}, J. Reine Angew. Math. \textbf{543} (2002), 147--168.

\bibitem{BPS98}
Dave Bayer, Irena Peeva, and Bernd Sturmfels, \emph{Monomial resolutions}, Math. Res. Lett. \textbf{5} (1998), no. 1--2, 31--46.

\bibitem{BSW}
Kristen~A. Beck and Sean Sather-Wagstaff, \emph{A somewhat gentle introduction to differential graded commutative algebra}, Connections Between Algebra, Combinatorics, and Geometry (New York, NY) (Susan~M. Cooper and Sean Sather-Wagstaff, eds.), Springer New York, 2014, pp.~3--99.

\bibitem{BE77}
David Buchsbaum and David Eisenbud, \emph{Algebra structures for finite free resolutions, and some structure theorems for ideals of codimension 3}, American Journal of Mathematics \textbf{99} (1977), 447--485.

\bibitem{CK24}
Trung Chau and Selvi Kara, \emph{Barile--{M}acchia resolutions}, J. Algebraic Combin. \textbf{59} (2024), no.~2, 413--472. \MR{4713508}

\bibitem{Ei80}
David Eisenbud, \emph{Homological algebra on a complete intersection, with an application to group representations}, Transactions of the American Mathematical Society \textbf{260} (1980), 35--64.

\bibitem{Eisenbud:2013}
David Eisenbud, \emph{Commutative algebra: with a view toward algebraic geometry}, vol. 150, Springer Science \& Business Media, 2013.

\bibitem{Eisenbud/Peeva:2016}
David Eisenbud and Irena Peeva, \emph{Minimal free resolutions over complete intersections}, vol. 2152, Springer, 2016.

\bibitem{Ge76}
Demissu Gemeda, \emph{Multiplicative structure of finite free resolutions of ideals generated by monomials in an {R}-sequence}, Ph.D. thesis, ProQuest LLC, Ann Arbor, MI, 1976, Thesis (Ph.D.)–Brandeis University. MR 2626146, 1976.

\bibitem{M2}
Daniel~R. Grayson and Michael~E. Stillman, \emph{Macaulay2, a software system for research in algebraic geometry}, Available at \url{https://math.uiuc.edu/Macaulay2/}.

\bibitem{Iyen97}
Srikanth~B. Iyengar, \emph{Free resolutions and change of rings}, Journal of Algebra \textbf{190} (1997), 195--213.

\bibitem{Kat19}
Lukas Katthän, \emph{The structure of dga resolutions of monomial ideals}, Journal of Pure and Applied Algebra \textbf{223} (2019), no.~3, 1227--1245.

\bibitem{Ly88}
Gennady Lyubeznik, \emph{A new explicit finite free resolution of ideals generated by monomials in an {R}-sequence}, J. Pure Appl. Alg. \textbf{51} (1988), 193--195.

\bibitem{Nasseh2021}
Saeed Nasseh and Keri Sather-Wagstaff, \emph{Applications of differential graded algebra techniques in commutative algebra}, pp.~589--616, Springer International Publishing, Cham, 2021.

\bibitem{PeevaThesis}
Irena Peeva, \emph{Strongly stable ideals}, Ph.D. thesis, Brandeis University, 1994.

\bibitem{Pe11}
\bysame, \emph{Graded syzygies}, Springer-Verlag London, Ltd., London, 2011.

\bibitem{Pollitz/Sega:2024}
Josh Pollitz and Liana~M Sega, \emph{Relations between poincar$\backslash$'e series for quasi-complete intersection homomorphisms}, arXiv preprint arXiv:2403.17079 (2024).

\bibitem{Roberts:1980}
Paul Roberts, \emph{Homological invariants of modules over commutative rings}, Sem. Math. Sup. \textbf{72} (1980).

\bibitem{Sha69}
Jack Shamash, \emph{The poincaré series of a local ring}, Journal of Algebra \textbf{12} (1969), no.~4, 453--470.

\bibitem{Sko11}
Emil Sk\"oldberg, \emph{Resolutions of modules with initially linear syzygies}, arXiv:1106.1913v2 (2011).

\bibitem{Sob23}
Aleksandra Sobieska, \emph{A {T}aylor resolution over complete intersections}, Journal of Algebra \textbf{636} (2023), 716--731.

\bibitem{Tate57}
John Tate, \emph{Homology of noetherian rings and local rings}, Illinois Journal of Mathematics \textbf{1} (1957), 14--27.

\bibitem{Tay66}
Diana~Kahn Taylor, \emph{Ideals generated by monomials in an {R}-sequence}, Ph.D. thesis, University of Chicago, Department of Mathematics, 1966.

\end{thebibliography}
\end{document}